\documentclass[11pt, twoside]{article}
\usepackage{amsmath,amsthm,amssymb}
\usepackage{times}  
\usepackage{enumerate} 

\usepackage{color}
\usepackage[colorlinks]{hyperref}   
 
\def\titlerunning#1{\gdef\titrun{#1}}
\makeatletter
\def\author#1{\gdef\autrun{\def\and{\unskip, }#1}\gdef\@author{#1}}
\def\address#1{{\def\and{\\\hspace*{18pt}}\renewcommand{\thefootnote}{}%
\footnote {#1}}%
\markboth{\autrun}{\titrun}}
\makeatother
\def\email#1{e-mail: #1}

\newtheorem{thm}{Theorem}[section]
\newtheorem{cor}[thm]{Corollary}
\newtheorem{lemma}[thm]{Lemma}
\newtheorem{prop}[thm]{Proposition}
\newtheorem{defn}[thm]{Definition}
\newtheorem{example}[thm]{Example}

\newtheorem{remark}[thm]{Remark}

\theoremstyle{definition}

\numberwithin{equation}{section}

\def\diam{\text{diam}}

\def\N{\mathbb{N}}

\def\R{\mathbb{R}}

\def\C{\mathbb{C}}
\def\CC{\mathcal{C}}
\def\P{\mathcal{P}}
\def\L{\mathcal{L}}

\def\H{\mathcal{H}}

\def\eps{\varepsilon}
\def\ie{{\em i.e.,\ }}

\def\dom{\text{dom}}
\def\range{\text{range}}
\def\BV{\text{BV}}
\def\Osc{\text{Osc}}
\def\Var{\text{Var}}
\def\Leb{\text{Leb}}
\def\var{\mathfrak{var}}
\def\Size{\text{Size }}

\def\twothird{\mbox{\small $\frac23$}}

\def\green{\color{green}}

\begin{document}

%%%%% To ease editing, add:

\baselineskip=17pt

%%%%%%%%%%%%%%%%

%% In the running head, give an abbreviation of the title. 
\titlerunning{The Dolgopyat inequality in \BV\ for non-Markov maps}

\title{The Dolgopyat inequality in bounded variation for non-Markov maps}

\author{Henk Bruin and Dalia Terhesiu}

\date{December 2016}%{\today}

\maketitle

\address{Faculty of Mathematics, University of Vienna, Oskar Morgensternplatz 1, 1090 Vienna, Austria; \email{henk.bruin@univie.ac.at}
\and
College of Engineering, Mathematics and Physical Sciences
Harrison Building
Streatham Campus, University of Exeter,
North Park Road, Exeter EX4 4QF, UK; 
\email{daliaterhesiu@gmail.com}}

\abstract{
Let $F$ be a (non-Markov) countably piecewise expanding interval map satisfying certain regularity conditions, 
and $\tilde\L$ the corresponding transfer operator. We prove the Dolgopyat inequality for the twisted operator
$\tilde\L_s(v) = \tilde\L_s(e^{s\varphi}v)$ acting on the space \BV of
functions of bounded variation, where $\varphi$ is a piecewise $C^1$ roof 
function.
}

\section{Introduction}
A crucial method (including what is now known as the Dolgopyat inequality) to prove exponential decay of correlations 
for Anosov flows with $C^1$ stable and unstable foliations was developed by Dolgopyat \cite{D}.
Liverani \cite{L} obtained exponential decay of correlations for Anosov flows with contact structure (and hence geodesic flow on compact negatively curved manifolds of any dimension).

Baladi \& Vall\'ee \cite{BalVal} further refined the method of \cite{D}
to prove exponential decay of correlations for 
suspension semiflows over one-dimensional piecewise $C^2$ expanding Markov  
maps with $C^1$ roof functions.
This was extended to the multidimensional setting by Avila et al.\ \cite{AGY},
to prove exponential decay of correlations of Teichm\"uller flows.
Ara\'ujo \& Melbourne \cite{AM} showed that the method can be adapted to
suspension semiflows over $C^{1+\alpha}$ maps with $C^1$ roof functions, 
which enabled them to prove that the classical 
Lorenz attractor has exponential decay of correlations.

In all of the above works, the results are applied to $C^\alpha$ observables for some $\alpha > 0$.
In this paper, we consider a class of non-Markov maps 
(see Section~\ref{sec:setup}), obtain a Dolgopyat inequality on the space 
of bounded variation (BV) observables (Theorem~\ref{th-main}).
The Dolgopyat inequality obtained in this paper automatically allows us to 
obtain exponential decay of correlations for skew-products on ${\mathbb T}^2$ 
as  considered by Butterley and Eslami \cite{BE,Eslami}, where the developed methods do 
not exploit the presence of the Markov structure.

Most probably, a proof of exponential decay for \BV\ observables for the 
class of non Markov maps considered here is not the easiest route; one could, 
for instance, think of inducing to a Markov map for which exponential decay of 
correlation of $C^2$ observables is known and then use approximation arguments 
to pass to \BV\ observables. Instead, we believe that the benefit of the Dolgopyat inequality  
in this setting is that it can be used to study perturbations of the flow (such as inserting 
holes in the Poincar\'e map); it is not at all clear that this can be economically done via inducing.

The main new ingredient of the proof is to locate and control the sizes of 
the jumps associated with \BV\ functions (see Section~\ref{sec-newingr}).

\subsection{Specific Examples}\label{sec:examples}
Our results (i.e., the Dolgopyat type inequality given by Theorem~\ref{th-main}) apply to 
typical AFU maps presented in Section~\ref{sec:setup}.
By typical we mean the whole clas of AFU maps (studied by Zweim\"uller \cite{Zwei0,Zwei})
satisfying assumption \eqref{eq:k} below. This assumption is very mild, see Remark~\ref{rem:typical}.
In particular, this class contains some standard families, 
such as the shifted
$\beta$-transformations $F:[0,1] \to [0,1]$, $x \mapsto \beta x + \alpha \pmod 1$ for fixed $\alpha \in [0,1)$ 
and $\beta > 1$.

Another important example is the First Return Map of a (non-Markov) Manneville-Pomeau map.
That is, 
$$
F = f^\tau:[\frac12,1] \to [\frac12,1]\quad \text{ for }\quad
\tau(x) = \min\{ n \geq 1 : f^n(x) \in [\frac12,1]\},
$$
where
$$
f:[0,1] \to [0,1], \quad x \mapsto \begin{cases}
        x(1+2^\alpha x^\alpha) & x \in [0,\frac12); \\
        \gamma(2x-1) & x \in [\frac12,1],
       \end{cases}
$$
is a non-Markov Manneville-Pomeau map
with fixed $\alpha > 0$ and {\green $\gamma \in (\frac12,1]$}.

The assumptions below apply to these to these examples, albeit that \eqref{eq:k}
holds for all parameters with the exception of a set of Hausdorff dimension $< 1$, 
see Remark~\ref{rem:typical}. The UNI condition
\eqref{eq:UNI} is a generic condition on the roof function 
of the type previously considered in \cite{BalVal, AGY}.

\section{Set-up, notation, assumptions and results.}\label{sec:setup}

We start this section by discussing
the class of AFU maps studied by Zweim\"uller \cite{Zwei0, Zwei}.
We present their conditions in Subsections~\ref{sec:AFU}-\ref{subsec-standhyp}.

\subsection{The AFU map $F$.}\label{sec:AFU}
Let $Y$ be an interval and $F:Y \to Y$ a topologically mixing
piecewise $C^2$ AFU map (\ie uniformly expanding
with finite image partition and satisfying Adler's condition), preserving
a probability measure $\mu$ which is absolutely continuous w.r.t.\
Lebesgue measure $\Leb$.
Let $\alpha$ be the partition of $Y$ into domains of the branches of $F$,
and $\alpha^n = \bigvee_{i=0}^{n-1} F^{-i}\alpha$.
Thus $F^n:a \to F^n(a)$ is a monotone diffeomorphism for each $a \in \alpha^n$.
The collection of inverse branches of $F^n$ is denoted as $\H_n$,
and each $h \in \H_n$ is associated to a unique $a \in \alpha^n$
such that $h:F^n(a) \to a$ is a contracting diffeomorphism.

\subsection{Uniform expansion.}
Let
\begin{equation}\label{eq:rho}
\rho_0 = \inf_{x \in Y} |F'(x)| \quad \text{ and } \quad
\rho = \rho_0^{1/4}.
\end{equation}
Since $F$ is uniformly expanding, $\rho_0 > \rho > 1$,
but in fact, we will assume that $\rho_0 > 2^{4/3}$,
which can be achieved by taking an iterate.

\subsection{Adler's condition.}
This condition states that
$\sup_{a \in \alpha} \sup_{x \in a} \frac{|F''(x)|}{|F'(x)|^2} < \infty$.
As $F$ is expanding, $\frac{|(F^n)''(x)|}{|(F^n)'(x)|^2}$ is 
bounded uniformly over the iterates 
$n \geq 1$, $a \in \alpha^n$ and $x \in a$ as well.
Thus, there is $C_1 \geq 0$ such that
\begin{equation}\label{eq:adler}
\frac{|(F^n)''(h(x))|}{|(F^n)'(h(x))|^2} \leq C_1 
\quad \text{ and } \quad 
\frac{h'(x)}{h'(x')} \leq e^{C_1 |x - x'|}
\end{equation}
for all $n \geq 1$, $h \in \H_n$ and $x, x' \in \dom(h)$.
The second inequality follows from the first by a standard computation.

\subsection{Finite image partition.}\label{sec:finite_image}
The map $F$ need not preserve a Markov partition, but has the finite image 
property. Therefore $K := \min\{ |F(a)| : a \in \alpha\}$ is positive.
We assume that $F$ is topologically mixing.
This implies that there is $k_1 \in \N$ such 
that $F^{k_1}(J) \subset Y$ for all intervals $J$ of 
length $|J| \geq \delta_0 := \frac{K(\rho_0-2)}{5e^{C_1}\rho_0}$
(this choice of $\delta_0$ is used in Lemma~\ref{lem:eta1}).

Let $X_1 = X'_1$ be the collection of boundary points of $F(a)$, $a \in \alpha$,
where $\alpha$ is the partition of $Y$ into branches of $F$.
Due to the finite image property, 
$X_1$ is a finite collection of points; we denote its cardinality by $N_1$.
Inductively, let $X'_k = F(X'_{k-1})$, 
\ie  the set of ``new'' boundary points of the $k$-th image partition,
 and $X_k = \cup_{j \leq k} X'_j$.
Therefore $\# X'_k \leq k N_1$.
Let $\{ \xi_i \}_{i=0}^M$ be a collection of points
containing $X_k$, and put in increasing order,
Then 
$$
\P_k =\{ (\xi_{i-1}, \xi_i) : i = 1, \dots ,M \}
$$ 
is a partition of $Y$, refining the {\em image partition of $F^k$}.
In other words, the components of 
$Y \setminus \{ \xi_i \}_{i=0}^M$ are the atoms of $\P_k$.

\subsection{Roof function.}
Let $\varphi:Y \to \R^+$ be a piecewise $C^1$ function, such that 
 $\varphi \geq 1$ and
\begin{equation}\label{eq:C2}
C_2 := \sup_{h \in \H} \sup_{x \in \dom(h)} |(\varphi \circ h)'(x)| < \infty.
\end{equation}
Since a main application is the decay of correlations of the vertical suspension semi-flow on
$\{ (y,u) : y \in Y, 0 \leq u \leq \varphi(y)\}/(y, \varphi(y)) \sim (F(y),0)$,
see Subsection~\ref{sec:semiflow},
we will call $\varphi$ the {\em roof function}.

Also assume that there is  $\eps_0 > 0$ such that
\begin{equation}\label{eq:sum}
C_3 := \sup_{x \in Y} \sum_{h \in \H, x \in \dom(h)} |h'(x)| e^{\eps_0 \varphi \circ h(x)} < \infty.
\end{equation}

\subsection{Further assumption on $F$ (relevant for the non-Markov case)}\label{subsec-standhyp}
We first discuss some known properties of the transfer operator
and twisted transfer operator.
Let $\Leb$ denote Lebesque measure.
Define the \BV-norm $\| v \|_{\BV}$ of $v:I \to \C$, for an  interval $I \subset \R$, as the sum of 
its $L^1$-norm (w.r.t. $\Leb$) $\| v \|_1$ and the total
variation $\Var_I v = \inf_{\tilde v = v\text{ a.e.}} 
\sup_{x_0 < \dots < x_N \in I} \sum_{i=1}^N |\tilde v(x_i)-\tilde v(x_{i-1})|$. 

Let $\L:L^1(Y,\Leb)\to L^1(Y,\Leb)$ be the transfer operators associated to $(Y, F)$
given by $\L^n v= \sum_{h\in\H_n} |h'| v\circ h$, $n\geq 1$.
For $s = \sigma + ib \in \C$, let $\L_s$ be the twisted version of 
$\L$ defined via $\L_s v= \L(e^{s\varphi} v)$ with iterates 
\begin{equation*}
\L_s^n v= \sum_{h\in\H_n} e^{s\varphi_n\circ h} |h'| v\circ h, \quad n\geq 1.
\end{equation*}
We first note that for $s = \sigma \in \R$,
\begin{prop}\label{prop-bnorm}
There exist $\eps\in (0,1)$ such that for  all $|\sigma|<\eps$,
$\|\L_{\sigma} \|_{\BV}<\infty$.
\end{prop}
\begin{proof} By Remark~\ref{rem-LYtwist}, there exist $c_1, c_2>0$ and $\eps\in (0,1)$ such that
$\Var_Y (\L_{\sigma} v) \leq c_1 \Var_Y v +c_2 \|v\|_\infty$, for all $|\sigma|<\eps$. Note that for any $v \in \BV(Y)$, $\|v\|_\infty\leq \Var_Y v+\|v\|_1$. Hence,
$\Var_Y (\L_{\sigma} v) \leq (c_1+c_2)\Var_Y v +c_2\|v\|_1$. Also, $\int_Y |\L_{\sigma} v|\, d\Leb\leq C_2\|v\|_\infty\leq C_2(\Var_Y v+\|v\|_1)$
and the conclusion follows.
\end{proof}

It is known that $\L_0=\L$ has a simple eigenvalue $\lambda_0=1$ with eigenfunction 
 $f_0 \in \BV$, \cite[Lemma 4]{Zwei0} (see also \cite{Rychlik}), and
$\frac{1}{C_4} \leq  f_0(x) \leq C_4$ for all $x \in Y$, see \cite[Lemma 7]{Zwei}.  Hence, $f_0$ is bounded from above and below.
This together with Proposition~\ref{prop-bnorm} implies that there exists $\eps\in (0,1)$ such that  $\L_{\sigma}$ has a family of simple eigenvalues 
$\lambda_{\sigma}$ for $|\sigma|<\eps$ 
with \BV\ eigenfunctions $f_{\sigma}$. 

We assumed above that $F$ has the finite image property, 
but not that $F^n$ has the finite image property uniformly over
$n \geq 1$.
We put a condition on $F$ as follows: the lengths of the atoms $p \in \P_k$, with $k$ specified below, 
do not decrease faster than $\rho^{-k}$:
\begin{equation}\label{eq:k}
\min_{p \in \P_k} \Leb(p) > \frac{16 C_8}{C_9}\frac{\sup f_{\sigma}}{\inf f_{\sigma}} \rho^{-k},
\end{equation} 
where $C_8 = 3C_7/\eta_0$ with $\eta_0:=(\sqrt 7-1)/2$ and 
$C_7 \geq 1$ is as in Lemma~\ref{lem:jump_fsigma}, and
$C_9$ is as in Lemma~\ref{lem:G}. 
Note that $\frac{\sup f_{\sigma}}{\inf f_{\sigma}}<\infty$ for $|\sigma|$ small (see Remark~\ref{rem:lambda}).

\begin{remark}\label{rem:typical}
Assumption~\eqref{eq:k} is trivially satisfied if $F$ is Markov.
For many one-parameter families of non-Markov AFU maps, one can show that 
\eqref{eq:k} only fails at a parameter set of Hausdorff dimension $< 1$.
This follows from the shrinking targets results \cite[Theorem 1 and Corollary 1]{AP} and includes the 
family of shifted $\beta$-transformations $x \mapsto \beta x + \alpha \bmod 1$.
%This is the {\lq}typical{\rq} mentioned in Subsection~\ref{sec:examples}.
\end{remark}

Throughout we fix $k\geq 2k_1$ sufficiently large to satisfy:
\begin{equation}\label{eq:kE}
\rho^k (\rho-1) > 12 N_1 C_8,
\end{equation}
(Inequality \eqref{eq:kE} will be used in estimates in Section~\ref{sec-disc-jumps}.)
Furthermore, we assume that
\begin{equation}\label{eq:k3}
\rho^{-2k}(\sup f_0 +\Var f_0 ) \Big(\frac{1}{\inf f_0 }+\Var\Big(\frac{1}{f_0 }\Big)\Big)<1,
\end{equation}
where $f_0$ is the positive eigenfunction of $\L_0$ associated
to eigenvalue $\lambda_0 = 1$.

\subsection{UNI condition restricted to atoms of the image partition $\P_k$}
\label{subsec-uni}
Fix $k$ as in Subsection~\ref{subsec-standhyp}. 
Let $C'_2 := \frac{C_2 \rho_0}{\rho_0-1}$ and
$C_{10} := (C_1e^{C_1} + 2(1+\eps_0) e^{\eps_0 C'_2}C'_2 + 2C_6)/(2\eta_0-4\rho_0^{-k})$,
where it follows from \eqref{eq:kE} that the denominator 
$2\eta_0 - 4\rho_0^{-k} > 0$. 
We assume that there exist $D>0$ and a multiple $n_0$ of $k$ such that
both
\begin{equation}\label{eq-n0-1}
C_{10}\rho_0^{-n_0}\frac{4\pi}{D}\leq \frac{1}{4}(2-2\cos\frac{\pi}{12})^{1/2},
\end{equation}
and the UNI (uniform non-integrability) condition holds:
\begin{equation}\label{eq:UNI}
\forall \text{ atom } p \in \P_k,  \ \exists h_1, h_2 \in \H_{n_0}\
\text{ such that } \inf_{x \in p} |\psi'(x)| \geq D,
\end{equation}
for $\psi = \varphi_{n_0} \circ h_1 - \varphi_{n_0} \circ h_2: p \to \R$.

\subsection{Main result}
Let $b\in\R$. For the class of \BV\ functions we define
\begin{equation}\label{eq-bnorm}
\|v\|_b= \frac{\Var_Y v }{1+|b|} + \|v\|_1.
\end{equation}
With the above specified, we can state our main result, a Dolgopyat type inequality.

\begin{thm}\label{th-main}
Suppose that all the above assumptions, ~\eqref{eq:rho} -- \eqref{eq:UNI}, 
on the AFU map $F$, on $k$ and on the roof function $\varphi$ hold
(in particular, we assume that UNI~\eqref{eq:UNI} hold for some $D>0$).
Then there exists $A\geq n_0$ and $\eps, \gamma<1$ such that
for all $|\sigma| < \eps$ and $|b| > \max\{4\pi/D, 2\}$ and for all $n\geq A \log|b|$,
$$
\| \L_s^n\|_{b}\leq \gamma^n.
$$
\end{thm}

An immediate consequence of the above result (see, for instance,~\cite{BalVal}) is
\begin{cor}
\label{cor-main}
Suppose that all the above assumptions, ~\eqref{eq:rho} -- \eqref{eq:UNI}, 
on the AFU map $F$, on $k$ and on the roof function $\varphi$ hold.
For every $0<\alpha<1$ there exists $\eps\in (0,1)$ and $b_0>0$ such that for all $|b|\geq b_0$ and for all $|\sigma|<\eps$,
$$
\| (I - \L_s)^{-1}\|_{b} \leq |b|^\alpha.
$$
\end{cor}

\begin{remark} A similar, but simplified, argument (obtained by taking $\sigma=0$ throughout the proof of Theorem~\ref{th-main} in this paper)
shows that without assuming condition \eqref{eq:sum} (that guarantees exponential tail for the roof function
$\varphi$) and with no restriction on the class of \BV\ functions,
one obtains that  for every $0<\alpha<1$, there exists $b_0>0$ such that for all $|b|\geq b_0$,
$\| (I - \L_{ib})^{-1} \|_{b} \leq |b|^\alpha$. Of course, this type of inequality does not imply exponential decay of correlation 
for suspension semiflows, but we believe it to be useful when proving sharp mixing rates for $\BV$ observables
in the non exponential situation via renewal type arguments (such
as sharp bounds for polynomial decay of correlation).
\end{remark}

\subsection{Application to suspension semi-flows}\label{sec:semiflow}
Corollary~\ref{cor-main} can be used to obtain exponential decay of correlations in terms of \BV\ functions 
for suspension semiflows  over AFU maps with a $C^1$ roof function.
Let $Y^\varphi:=\{(y,u)\in Y\times R:0\leq u\leq R(y)\}/\!\!\!\sim$,
where $(y,\varphi(y))\sim (F y,0)$, be the suspension over $Y$. The suspension semiflow  $F_t:Y^\varphi\to Y^\varphi$ is defined by $F_t(y,u)=(y,u+t)$ computed modulo identifications.
The probability measure $\mu^\varphi:=(\mu\times Leb)/\bar\varphi$, 
where $ \bar\varphi:=\int_Y\varphi d\mu$ is $F_t$-invariant. 
\\[3mm]
\paragraph{Class of observables}

Let  $F_{\BV,m}(Y^\varphi)$ be the class of observables consisting of $v(y,u) : Y^\varphi\to\C$ such that
$v$ is $\BV(Y)$ in $y$ and $C^m$ in $u$, so
$\|v\|_{\BV, m}:=\sum_{j=0}^m \| \partial_t^j  v \|_{\BV}<\infty$.

For $v\in L^1(Y^\varphi)$ and $w\in L^\infty(Y^\varphi)$ define the correlation function
\[
\rho_{t}(v,w):=\int_{Y^\varphi} v w\circ F_t\, d\mu^\varphi-\int_{Y^\varphi} v \,d\mu^\varphi \int_{Y^\varphi} w \,d\mu^\varphi.
\]

The result below gives exponential decay of correlation
for $v\in F_{\BV,2}(Y^\varphi)$ and $w\in L^\infty(Y^\varphi)$.
It is likely that this also follows by reinducing $F$
to a Gibbs-Markov AFU map, to which \cite{BalVal,AM} apply, together with an approximation argument of \BV\ functions by $C^2$ functions.
However, it is worthwhile to have the argument for the original map $F$, 
for instance in situations where reinducing is problematic, such
as for families of open AFU maps with shrinking holes.

\begin{thm}\label{thm-decay}
Suppose that all the above assumptions, ~\eqref{eq:rho} -- \eqref{eq:UNI}, 
on the AFU map $F$ and the roof function $\varphi$ hold. 
Then there exist 
constants $a_0, a_1>0$ such that
\[
|\rho_{t}(v,w)|\leq a_0 e^{-a_1 t} \|v\|_{\BV, 2}\|w\|_{\infty},
\]
 for all $v\in F_{\BV,2}(Y^\varphi)$ and $w\in L^\infty(Y^\varphi)$.
\end{thm}

The proof of Theorem~\ref{thm-decay} is given in Appendix~\ref{sec:expo}.
Corollary~\ref{cor-main} also implies exponential decay of correlations in terms of \BV\ functions 
for skew products on $\mathbb{T}^2$ as considered in~\cite{BE, Eslami}. We note, however, that
the strength of Corollary~\ref{cor-main} is not needed in the set-up of ~\cite{BE, Eslami}
as, in those works, the roof function is bounded and one can restrict the calculations 
to the imaginary axis.

\section{Twisted and normalized twisted transfer operators}
\label{sec-twisted-normalized}

We start with the continuty of operator $\L_s$ in $\BV$.

\begin{prop}\label{prop-continuity}
Let $\eps_0 > 0$ and  $C_3<\infty$ be as in ~\eqref{eq:sum}. Then there exists $C>0$ and $\eps \in (0, \eps_0)$ such that for all 
$|\sigma_1|,|\sigma_2|<\eps$ and for all $|b_1|, |b_2| \leq 1$,
$\|\L_{\sigma_1+ib_1}-\L_{\sigma_2+ib_2}\|_{\BV}\leq C\eps_0^{-1}|\sigma_1-\sigma_2|$.
\end{prop}

The proof of  Proposition~\ref{prop-continuity} is deferred to the end of Appendix~\ref{sec-LY}.

\begin{remark}\label{rem:lambda}
An immediate consequence of Proposition~\ref{prop-continuity} is that for any $\delta\in(0,1)$, there exists $\eps\in (0,1)$
such that
\[
\sup_{|\sigma|<\eps}|\lambda_{\sigma}-1|<\delta,\quad \sup_{|\sigma|<\eps}
\|\frac{f_{\sigma}}{f_0}-1\|_{\BV}<\delta, \quad \sup_{|\sigma|<\eps}\|\frac{f_{\sigma}}{f_0}-1\|_{\infty}<\delta
\]
for all $|\sigma| < \eps$. 
Recall that $\frac{1}{C_4} \leq  f_0(x) \leq C_4$ for all $x \in Y$.
It follows that $\frac{f_{\sigma}(x)}{f_{\sigma}(y)}
=  \frac{f_{\sigma}(x)}{f_0(x)} \frac{f_0(x)}{f_0(y)} \frac{f_0(y)}{f_{\sigma}(y)} \leq (1+\delta) C^2 (1-\delta)^{-1} < \infty$ for all $x,y \in Y$.
Hence, $\frac{\sup f_{\sigma}}{\inf f_{\sigma}} \leq C_{5}$ for $C_{5} := \frac{1+\delta}{1-\delta} C_4^2$ and $|\sigma| < \eps$.
\end{remark}

Since $\lambda_0 = 1$ and $f_0$ is strictly positive, 
due to the continuity of $\lambda_{\sigma}$ and $f_{\sigma}$ in $\sigma$, 
we can ensure that for $\eps>0$ sufficiently small
\begin{equation}\label{eq-lamfsigma}
\rho^{-1/4} <\lambda_{\sigma}  \mbox{ and } f_{\sigma} \mbox{ is strictly
positive for all } |\sigma|<\eps.
\end{equation}
By assumption~\eqref{eq:k3} and Remark~\ref{rem:lambda}, we can choose $\eps$  small enough such that for all $|\sigma| < \eps$,
\begin{equation}\label{eq:rho3}
\rho^{-2k}(\sup f_{\sigma}+\Var f_{\sigma}) \Big(\frac{1}{\inf f_{\sigma}}+\Var\Big(\frac{1}{f_{\sigma}}\Big)\Big)<1.
\end{equation}
(The above formula will be used in the proof of Proposition~\ref{prop-LYineq}.)

\begin{lemma}\label{lem:eps}
There exists $\eps\in (0,1)$ so small that for all $|\sigma| < \eps$ 
and for all $n\geq 1$,
\begin{equation}\label{eq:rho2}
\frac{1}{\lambda_{\sigma}^n}
\sup_{h \in \H_n} \sup_{x \in \dom(h)} |h'(x)| e^{\sigma \varphi_n \circ h(x)}  \leq \rho^{-3n}.
\end{equation}
\end{lemma}

\begin{remark}
Without assumption \eqref{eq:sum} (i.e., without the exponential tail assumption), we still have 
$$
\sup_{h \in \H_n} \sup_{x \in \dom(h)} |h'(x)| e^{\sigma \varphi_n \circ h(x)}  \leq \rho^{-3n}
$$ 
for $-\eps < \sigma \leq 0$. 
\end{remark}

\begin{proof} We start with $n = 1$.
By continuity of $\lambda_{\sigma}$, we can take $\eps$ so small
that $\lambda_{\sigma}^{4u} \rho_0^{u-1} > C_3$ for $u = \lfloor \eps_0/(4\eps)\rfloor$ with $\eps_0\in (0,1)$ and $C_3$ such that \eqref{eq:sum} hold.
For $h \in \H_1$ assume by contradiction that
$\lambda_{\sigma}^{-1} |h'(x)| e^{\sigma \varphi \circ h(x)} > \rho^{-3}$
for some $x \in \dom(h)$.
Since $|h'| \leq \rho_0^{-1} = \rho^{-4}$ we have
$$
\lambda_{\sigma}^{-1} e^{\sigma \varphi \circ h(x)}
\geq \lambda_{\sigma}^{-1} \rho^4 |h'| e^{\sigma \varphi \circ h(x)} >
\rho = \rho_0^{1/4} \geq |h'|^{-1/4}.
$$
Therefore,
\begin{eqnarray*}
|h'| e^{\eps_0 \varphi \circ h} &>& |h'| e^{4u\eps \varphi \circ h} 
\geq |h'| e^{4u\sigma  \varphi \circ h} 
\geq |h'| (\lambda_{\sigma}^{-1} e^{\sigma  \varphi \circ h} )^{4u} \lambda_{\sigma}^{4u}\\
&\geq& |h'|^{1-u} \lambda_{\sigma}^{4u} \geq \rho_0^{u-1} \lambda_{\sigma}^{4u} \geq C_3
\end{eqnarray*}
contradicting \eqref{eq:sum}. The statement for $n \geq 1$
follows immediately.
\end{proof}

Let 
$$
\tilde \L_s v = \frac{1}{\lambda_{\sigma} f_{\sigma}} \L_s( f_{\sigma} v)
\quad \text{ and }\quad
\tilde \L_{\sigma} v = \frac{1}{\lambda_{\sigma} f_{\sigma}} \L_{\sigma}( f_{\sigma} v)
$$
be the \emph{normalized} versions of
$\L_s$ and $\L_{\sigma}$.

\begin{prop}[Lasota-Yorke type inequality] \label{prop-LYineq}
Choose $k$ and $\eps_1\in (0,1)$ such that
\eqref{eq:rho3} and ~\eqref{eq:rho2} hold.
Define $\Lambda_{\sigma} = \lambda_{2\sigma}^{1/2} / \lambda_{\sigma}$.
Then, there exist $\eps \leq \eps_1$, $\rho> 1$ and $c>0$  such that
for all $s=\sigma+ib$ with $|\sigma|<\eps$ and $b\in\R$,
\[
\Var_Y (\tilde \L_s^{nk} v) 
\leq\rho^{-nk} \Var_Y v + c(1+ |b|)\Lambda_{\sigma}^{nk}(\|v\|_\infty\|v\|_1)^{1/2}.
\]
for all $v\in \BV(Y)$ and all $n\geq 1$.
\end{prop}
Proposition~\ref{prop-LYineq} would be meaningless if $\Lambda_\sigma < 1$,
but one can check that $1 \leq \Lambda_\sigma = 1 + O(\sigma^2)$.
The proof of Proposition~\ref{prop-LYineq} is deferred to Appendix~\ref{sec-LY}.

In what follows we focus on the controlling the term  containing $(\|v\|_\infty\|v\|_1)^{1/2}$ and proceed as in \cite{BalVal}: we estimate the $L^2$ norm
of $\tilde \L_s^n$ for $n$ large enough. Once we obtain a good estimate for the $L^2$ norm, we combine it with the estimate in
Proposition~\ref{prop-LYineq} (following the pattern in \cite{AM, AGY, BalVal}) to prove Theorem~\ref{th-main}.

\section{New ingredients of the proof}
\label{sec-newingr}

The basic strategy of the proof using the cancellation lemma follows
\cite{AM, AGY, BalVal}.
For the non-Markov AFU maps, we use the space $\BV$, and hence 
observables $u,v \in \BV$ can have jumps.
The task is to locate and control the sizes of these jumps.
Given a discontinuity point $x$ for a function $v$, we define the {\em size of the jump} at
$x$ as
\begin{equation}\label{eq:jumpsize}
\Size v(x) = \lim_{\delta \to 0} \sup_{\xi, \xi' \in (x-\delta, x+\delta) } |v(\xi)-v(\xi')|.
\end{equation}
Recall that the oscillation of a function $v:I \to \C$ on a subinterval $I \subset Y$ is defined as
$$
\Osc_I v = \sup_{\xi, \xi' \in I} |v(\xi) -v(\xi')|.
$$
It follows that
\begin{equation}\label{eq:oscsize}
\Osc_I v \leq \Osc_{I^\circ} v + \Size v(x) + \Size v(y)
\end{equation}
for $I = [x,y]$ with interior $I^\circ$.
For positive functions, \eqref{eq:jumpsize} reduces to
\begin{equation}\label{eq:jumpsize2}
\Size u(x) = \limsup_{\xi \to x} u(\xi) - \liminf_{\xi \to x} u(\xi)
= |\lim_{\xi \uparrow x} u(\xi) - \lim_{\xi \downarrow x} u(\xi)|.
\end{equation}
We adopt the convention $u(x) =\limsup_{\xi \to x} u(\xi)$ at discontinuity points, so we always have the trivial inequality $\Size u(x) \leq u(x)$.

\begin{defn}\label{def:expjumpsize}
Let $k\geq 1$ such that~\eqref{eq:k} holds
and take $C_7$ as in Lemma~\ref{lem:jump_fsigma}.
We say that a pair of functions $u,v \in \BV(Y)$ with $|v| \leq u$ and $u>0$ has 
{\em exponentially decreasing jump-sizes}, if  
the discontinuities of $u$ and $v$
belong to $X_\infty = \cup_{j\geq 1} X'_j$ and if $x \in X'_j$ for $j > k$
is such a discontinuity, then 
\begin{equation}\label{eq:EJ}
\Size v(x) , \Size u(x) \leq C_7 \rho^{-j} u(x).
\end{equation}
\end{defn}

\begin{example} For the reader's convenience, we provide a simple example of functions $(u,v)$ with exponentially decreasing jump-sizes.
Assume that $Y = [p,q]$.
Let $\{a_i\}_{i \geq 1}$ be a sequence in $\C$ such that $|a_i| \to 0$ exponentially fast, and $\{x_i\}_{i \geq 1} \subset [p, q]$.
Then
$$
v = \sum_{i \geq 1} a_i 1_{[x_i,q]} \qquad  u =  \sum_{i \geq 1} |a_i| 1_{[x_i,q]}
$$
is a pair of functions having exponentially decreasing jump-sizes
when $X'_j = \{ x_j\}$.
Indeed, let $\delta' > 0$ be arbitrary and let $N \in \N$ be such that
$\sum_{i > N} |a_i| \leq \delta'$.
Assuming for simplicity that the $x_i$ are distinct, we have
\begin{align*}
\Size v(x_j) & = \lim_{\delta \to 0} \sup_{\xi, \xi' \in (x_j-\delta, x_j+\delta)} 
\left| \sum_{i\geq 1} a_i \Big(1_{[x_i,q]}(\xi) - 1_{[x_i,q]}(\xi')\Big) \right| \\
& \leq \lim_{\delta \to 0} \sup_{\xi, \xi' \in (x_j-\delta, x_j+\delta)} 
\left| \sum_{i=1}^N a_i\Big( 1_{[x_i,q]}(\xi) - 1_{[x_i,q]}(\xi')\Big) \right|
+ \delta' = |a_j| +  \delta'.
\end{align*}
Since $\delta'$ was arbitrary, $\Size v(x_j) \leq |a_j|$.
So, $\Size v(x_j)$
is exponentially small in $j$.
On the other hand, if $x \notin \{ x_i\}_{i \in \N}$, then $v$ is continuous at $x$, so $\Size v(x) = 0$.
A similar computation holds for $\Size u(x_j)$.
\end{example}

Definition~\ref{def:expjumpsize} states that the discontinuities of $(u,v)$ can only appear in $X_\infty := \cup_{j \geq 1} X'_j$, and we will see in 
Proposition~\ref{prop:inductive_jumpsize} that this property is preserved
under $(u,v) \mapsto (\tilde\L_{\sigma}^n u, \tilde\L_s^n v)$.
For a given $n$, we will distinguish between two types of discontinuities
of $\tilde \L_{\sigma}^n u$.\\
{\bf (i)} {\em Created} discontinuities.
In this case $x \in \partial\dom(h)$ for some
$ h \in \H_n$ and $x \in X'_j$ for some $1 \leq j \leq n$. 
The discontinuity is created because 
the sum $\sum_{h \in \H, \xi \in \dom(h)}$ involved
in $\tilde \L_{\sigma}^n u$ runs over a different collection of inverse branches depending
on whether $\xi$ is close to the left or close to the right of $x$:
in only one of the cases $h$ is part of this collection.
It is not important whether the function $u$ is continuous
at $y = h(x)$.\\
{\bf (ii)} {\em Propagated} discontinuities.
Here the function $u:Y\to\R_{+}$ has discontinuities. Hence, it is discontinuous at $y = h(x)$ for some $h \in \H_n$.
In this case $y \in X'_j$ for some $j \geq 1$ and hence $x \in X'_{j+n}$.

Consequently, we define a cone $\CC_b$ of $\BV$ functions with
discontinuities  of the type prescribed in Definition~\ref{def:expjumpsize}.
In Appendix~\ref{sec:Hofbauer},
we prove that the eigenfunction $f_{\sigma}$ and $1/f_\sigma$
belong to $\CC_b$.
This argument is independent of Section~\ref{sec-invcone}
where the invariance of $\CC_b$ under the transformation
$(u,v) \mapsto (\tilde \L_{\sigma}^n(\chi u), \tilde \L_s^nv)$ is proved.
This invariance depends crucially on
Proposition~\ref{prop:inductive_jumpsize}, which
 together with an inductive bound
on $\frac{\sup u|_p}{\inf u|_p}$
for $p \in \P_k$ and assumption 
\eqref{eq:k} imply that discontinuities indeed behave as outlined 
in this section.
To deal with  \BV\ observables $v \notin \CC_b$, we exploit
the fact that
the size of
discontinuities at points $x \notin X_\infty$ decrease
exponentially under iteration of $\tilde \L_s$.
This means that $\tilde \L_s^nv$ converges exponentially
fast to $\CC_b$ and this suffices to prove the results for 
arbitrary \BV\ observables.

\section{Towards the cone condition: discontinuities and jump-sizes}
\label{sec-disc-jumps}

Recall the sets $X'_j$ from Section~\ref{sec:finite_image}
and let $k$ satisfy the conditions in Subsection~\ref{subsec-standhyp}.
To deal with the discontinuities of $(u,v)$, we introduce 
the ``extra term'' for intervals $I \subset Y$:
\begin{equation}\label{eq:E}
E_I(u) := \sum_{j > k} \rho^{-j} \sum_{x \in X'_j \cap I^\circ} 
\limsup_{\xi \to x} u(\xi),
\end{equation}
where we recall that $\# X'_j \leq N_1$ for all $j \geq 1$.
The choice of $k$ in \eqref{eq:kE} implies that
$C_8 E_I(u) \leq \frac{1}{12} \sup_I u$ for every $I$ contained in a single atom of $\P_k$.

Throughout this and the next section we set $n = 2k$.
We start with two lemmas on the properties of the eigenfunction $f_{\sigma}$,
which will be proved in Section~\ref{sec:Hofbauer}. We recall (see Remark 1.4) that $f_{\sigma}$ is the positive eigenfunction of $\L_{\sigma}$  with eigenvalue $\lambda_{\sigma}$.

\begin{lemma}\label{lem:jump_fsigma}
There are $C_6, C_7 \geq 1$ 
such that for all $\sigma$ with $|\sigma|<\eps$
the following holds:
\begin{enumerate} 
\item $f_{\sigma}$ has discontinuities only in
$X_\infty$, and if $x_j \in X'_j$, then
$\Size f_{\sigma}(x_j) \leq C_7 \rho^{-3j} \sup f_{\sigma}$.
\item For every interval $I \subset Y$ we have
$$
\Osc_{I^\circ}(f_{\sigma}) \leq C_6 \Leb(I) \inf_I f_{\sigma}  +  C_7 E_I(f_{\sigma})
\ \text{ and } \
\Osc_{I^\circ}\Big(\frac{1}{f_{\sigma}}\Big) \leq C_6 \Leb(I)\inf_I \frac{1}{f_{\sigma}}  +
 C_7 E_I\Big(\frac{1}{f_{\sigma}}\Big).
$$
\end{enumerate}
\end{lemma}

\begin{lemma}\label{lem:G}
Choose $k$ such that \eqref{eq:k} holds and set $n = 2k$.
Then there exists $\eps \in (0, 1)$ and $C_9 \in (0, 1)$ such that
$$
\lambda_{\sigma}^{-n} \inf_{x \in Y} 
\sum_{\stackrel{h \in \H_n, x \in \dom(h)}{\range(h) \subset p}} 
|h'(x)| e^{\sigma \varphi_n \circ h(x)} \geq C_9 \Leb(p)
$$ 
for all $p \in \P_k$ and $|\sigma| < \eps$.
\end{lemma}

The main result in this section is the following.

\begin{prop}\label{prop:inductive_jumpsize}
Choose $k$ such that \eqref{eq:k} holds and set $n = 2k$.
If the pair $(u,v)$ with $|v| \leq u$ has exponentially decreasing 
jump-sizes \eqref{eq:EJ}, 
then for each $x \in X'_j$ with $j > k$, we have 
$$
\Size \tilde\L_{\sigma}^n u(x)\  ,\ \Size \tilde\L_s^n v(x)
\leq \frac14 \max_{p \in \P_k}\frac{\sup u|_p}{\inf u|_p}\  C_7 \rho^{-j} \tilde\L_{\sigma}^n u(x).
$$  
\end{prop}

\begin{remark}\label{rem:multiple} 
It is possible that $x$ belongs to different $X'_j$'s at the same time. This means that the discontinuity at $x$ is propagated by different branches of $F$ (or $x \in X'_1 \cap X'_j$ for some $j \geq 2$, and the discontinuity
at $x$ is generated in $\P_1$ as well as propagated from another discontinuity at some point in $X'_{j-1}$).
In this case, we add the jump-sizes at $x$ but the proof remains
the same, \ie writing $x = x_j = x_{j'}$ for $x_j \in X'_j$ and $x_{j'} \in X'_{j'}$,
$\Size v(x) = \Size v(x_j) + \Size v(x_{j'}) \leq C_7 (\rho^{-j} + \rho^{-j'}) \| u \|_\infty$.
\end{remark}

\begin{proof}[Proof of Proposition~\ref{prop:inductive_jumpsize}]
By Lemma~\ref{lem:jump_fsigma}, we know that $f_{\sigma}$ and $1/f_{\sigma}$ have exponentially decreasing jump-sizes with parameters $C_7$ and $\rho^3$.

Let $y = \tilde h(x)$ for some $\tilde h \in \H_r$ and $r > k$ to be determined below.
Let $p \in \P_k$ such that $y \in \overline{p}$.
Then
\begin{eqnarray}\label{eq:Lnu}
\tilde \L_{\sigma}^r u(x) &\geq &
\frac{1}{\lambda_{\sigma}^r f_{\sigma}(x)} \sum_{\stackrel{h \in \H_r}{\range(h) \subset p}}
|h'| e^{\sigma \varphi_r \circ h(x)} (f_{\sigma} u)\circ h(x) \nonumber \\
&\geq & 
\frac{\inf f_{\sigma}}{f_{\sigma}(x)} \frac{\inf u|_p}{\sup u|_p}  u(y) 
 \lambda_{\sigma}^{-r} \sum_{\stackrel{h \in \H_r, x \in \dom(h)}{\range(h) \subset p}} |h'(x)| e^{\sigma \varphi_r \circ h(x)} \nonumber \\
&\geq & \frac{\inf f_{\sigma}}{f_{\sigma}(x)} \frac{\inf u|_p}{\sup u|_p} C_9 \Leb(p) u(y) 
\end{eqnarray}
by Lemma~\ref{lem:G}.

First take $j > n$ and $x \in X'_j$, so $x$ is a discontinuity
propagated from some $y \in X'_{j-n}$.
Let $\tilde h \in \H_n$
such that $\tilde h(x) = y$ be the corresponding inverse branch.
This is the only inverse branch that contributes to
$\Size \tilde \L_s^n v(x)$.
We compute using \eqref{eq:rho2} and Lemma~\ref{lem:jump_fsigma},
\begin{eqnarray}\label{eq:jump1}
&\Size& \!\!\!\!\tilde \L_s^n v(x) =
\Size \Big( |\tilde h'| e^{s \varphi_n \circ \tilde h} 
\frac{(f_{\sigma} v) \circ \tilde h}{\lambda_{\sigma}^n f_{\sigma}}\Big) (x) \nonumber \\
&\leq& \frac{1}{\lambda_{\sigma}^n} |\tilde h'(x)| e^{\sigma \varphi_n \circ \tilde h(x)} 
\Big( \frac{|v(y)|}{f_{\sigma}(x)} \Size f_{\sigma}(y)
+  f_{\sigma}(y) |v(y)| \Size \frac{1}{f_{\sigma}}(x) +
\frac{f_{\sigma}(y)}{f_{\sigma}(x)} \Size v(y) \Big) \nonumber \\
&\leq& 4\rho^{-3n}  \frac{\sup f_{\sigma}}{f_{\sigma}(x)}
 u(y) \times \begin{cases}
C_7  \rho^{-(j-n)} & \text{ if } j-n > k,\\
1 &  \text{ if } j-n \leq k.
\end{cases}
\end{eqnarray}
This distinction is because \eqref{eq:EJ} only holds for $j-n > k$;
for $j-n \leq k$ we only have the trivial bound $\Size v(y) \leq u(y)$.
The factor $4$ is to account for the three terms in the penultimate line above;
in particular, $\Size v(y) \leq 2 u(y)$, so the factor $4$ appears despite the presence of just three terms.
Since $\rho^{-2n} \leq \rho^{-4k}$,
we have 
\begin{equation}\label{eq:sizen}
\Size \tilde \L_s^n v(x) \leq \frac{4\sup f_{\sigma}}{\rho^{3k} f_{\sigma}(x)}
C_7 \rho^{-j} u(y)
\end{equation}
in either case.

Combining \eqref{eq:sizen} and \eqref{eq:Lnu} for $y = \tilde h(x)$ and $r=n$, and using the bound on $\Leb(p)$ from \eqref{eq:k} we obtain
$$
\Size\tilde \L_s^n v(x) \leq  
\frac{ 4C_7 }{C_9 \rho^{3k} \Leb(p)}
\frac{\sup u|_p}{\inf u|_p} 
\frac{\sup f_{\sigma}}{\inf f_{\sigma}} \ \rho^{-j} 
\tilde \L_{\sigma}^n u(x) 
\leq \frac14 \frac{\sup u|_p}{\inf u|_p}\  C_7 \rho^{-j} 
\tilde \L_{\sigma}^n u(x).
$$
Now take $k < j \leq n$, so the discontinuity at $x \in X'_j$ is created
by non-onto branches of $F^n$, and there exist $y \in X'_1$
and an inverse branch $\tilde h \in \H_{j-1}$ such that
$y = \tilde h(x)$. 
Then, analogous to \eqref{eq:jump1},
\begin{eqnarray*}
\Size \tilde \L_s^n v(x) &=&
\Size \Big( |\tilde h'| e^{s \varphi_{j-1} \circ \tilde h} 
\frac{(f_{\sigma} v) \circ \tilde h}{\lambda_{\sigma}^n f_{\sigma}}\Big) (x) \\
&\leq& \frac{1}{\lambda_{\sigma}^n} |\tilde h'(x)| e^{\sigma \varphi_{j-1} \circ \tilde h(x)} 
\frac{4\sup f_{\sigma}}{f_{\sigma}(x)} u(y) \\
&\leq& \frac{\rho^{-3(j-1)}}{\lambda_{\sigma}^{n-j+1} } 
\frac{4\sup f_{\sigma}}{f_{\sigma}(x)} u(y)
\leq \frac{4C_7 \sup f_{\sigma}}{\rho^k f_{\sigma}(x)} \ \rho^{-j} u(y)
\end{eqnarray*}
because $C_7 \geq 1$, $k < j \leq n$ and $\lambda_{\sigma}^{-4} \leq \rho$
by \eqref{eq-lamfsigma}.
Combining this with \eqref{eq:Lnu} to bound $u(y)$
(but applied to $r=j$) and \eqref{eq:k} gives
$$
\Size\tilde \L_s^n v(x) \leq  
\frac{ 4C_7 }{C_9 \rho^k \Leb(p)}
\frac{\sup u|_p}{\inf u|_p} 
\frac{\sup f_{\sigma}}{\inf f_{\sigma}} \ \rho^{-j} 
\tilde \L_{\sigma}^n u(x) 
\leq \frac14 \frac{\sup u|_p}{\inf u|_p}\  C_7 \rho^{-j} 
\tilde \L_{\sigma}^n u(x),
$$
as before. The computations for $\tilde\L_{\sigma}^n u$ are the same.
\end{proof}

\section{Cancellation lemma}
\label{sec-cance}

We define a cone of function pairs $(u,v)$:
\begin{align} \label{eq:cone}
\CC_b = \Big\{ (u ,v) \ : & \ 0 < u \ , \ 0 \leq |v| \leq u\ , (u,v) 
\text{ has exponentially decreasing }
\nonumber\\
& \text{ jump-sizes } \eqref{eq:EJ} \text{ and }
\Osc_I v \leq C_{10} |b| \Leb(I)\sup u|_I  + C_8 E_I(u), \\
& \text{ for all intervals $I$ contained in a single atom of } \P_k 
\Big\}. \nonumber
\end{align}
Recall that the choice of $k$ in \eqref{eq:kE} implies that
$C_8 E_I(u) \leq \frac{1}{12} \sup_I u$ for every $I$ contained in a single atom of $\P_k$.
In Section~\ref{sec-invcone} we show that $\CC_b$ is 'invariant' in the sense of~\cite{BalVal}: see Lemma~\ref{lemma-invcone}.

In this section we provide a cancellation lemma for pairs of functions in $\CC_b$ similar to the one in~~\cite{BalVal}.
The statement and proof of Lemma~\ref{lemma-cancellation} below follows closely the pattern of  the statements and proofs of 
~\cite[Lemma 2.4]{BalVal} and ~\cite[Lemma 2.9]{AM}.
In this section, we abbreviate
$$
A_{s,h,n}= e^{s\varphi_n\circ h} |h'| v\circ h
$$ 
for $h \in \H_n$ and $\varphi_n = \sum_{j=0}^{n-1} \varphi \circ F^j$.

\begin{lemma}\label{lemma-cancellation} Fix $k$ such that~\eqref{eq:k} holds.
Recall that $\eta_0=\frac{\sqrt 7-1}{2}\in (2/3,1)$. 
Assume that the UNI condition in Subsection~\ref{subsec-uni} holds
(with constant $D>0$, $k$ fixed and $n_0\geq 1$). 

Set $\Delta=\frac{2\pi}{D}$. There exists $\delta\in (0,\Delta)$ such that the following hold for all $|\sigma|<\eps$, $|b|>2\Delta$ and for all $(u,v)\in\CC_b$:

Let $p\in\P_k$ and let $h_1, h_2 \in \H_{n_0}$ be the branches from UNI.
For every $y_0\in p$ there exists $y_1\in B_{\Delta/|b|}(y_0)$ such that one of the following inequalities holds on $B_{\delta/|b|}(y_1)$:
\begin{itemize}
\item[Case $h_1$.] $|A_{s,h_1,n_0}(f_{\sigma} v)+A_{s,h_2,n_0}(f_{\sigma} v)|\leq \eta_0 A_{\sigma,h_1,n_0}(f_{\sigma} u)+ A_{\sigma,h_2,n_0}(f_{\sigma} u)$.

\item[Case $h_2$.] $|A_{s,h_1,n_0}(f_{\sigma} v)+A_{s,h_2,n_0}(f_{\sigma} v)|\leq  A_{\sigma,h_1,n_0}(f_{\sigma} u)+\eta_0 A_{\sigma,h_2,n_0}(f_{\sigma} u)$.
\end{itemize}
\end{lemma}

\begin{proof} Choose $\delta \in (0, \Delta)$ sufficiently small such that
\begin{equation}\label{eq-delta}
\delta \frac{D}{16\pi}< \frac{1}{12}, \quad\quad C_0\delta<\frac{\pi}{6}.
\end{equation}
Let $y_0\in Y$. Note that for $m=1,2$,
\begin{align*}
\sup_{B_{\delta/|b|}(y_0)}|v\circ h_m|\leq \Osc_{B_{\delta/|b|}(y_0)}(v\circ h_m)+ \inf_{B_{\delta/|b|}(y_0)}|v\circ h_m|
+\Size v(B_{\delta/|b|}(y_0)).
\end{align*}
Since $(u,v)\in\CC_b$,
\begin{align*}
\sup_{B_{\delta/|b|}(y_0)}|v\circ h_m|\leq C_{10} \Leb(h_m(B_{\delta/|b|}(y_0)))|b|\sup_{B_{\delta/|b|}(y_0)}(u\circ h_m)&+
 \inf_{B_{\delta/|b|}(y_0)}|v\circ h_m|\\
&+ C_8 E_{B_{\delta/|b|}(y_0)}(u).
\end{align*}
But
\[
C_{10}\Leb(h_m(B_{\delta/|b|}(y_0)))\leq C_{10} \rho_0^{-n_0}\Leb(B_{\delta/|b|}(y_0))=C_{10}\rho_0^{-n_0}\frac{\delta}{|b|}\leq \frac{D}{16\pi}\frac{\delta}{|b|},
\]
where in the last inequality we have used~\eqref{eq-n0-1}.
Putting the above together with the estimate on $E_I(u)$ below equation~\eqref{eq:E}
and using the choice of $\delta$ and $k$,
\begin{align}\label{eq-1}
 \sup_{B_{\delta/|b|}(y_0)}|v\circ h_m|
\leq  \frac{1}{6} \sup_{B_{\delta/|b|}(y_0)}(u\circ h_m)+ \inf_{B_{\delta/|b|}(y_0)}|v\circ h_m|.
\end{align}
{\bf Case 1.} Suppose that
$\inf_{B_{\delta/|b|}(y_0)}|v\circ h_m|\leq  \frac{1}{2} \sup_{B_{\delta/|b|}(y_0)}(u\circ h_m)$ for $m=1,2$. Then~\eqref{eq-1}
implies that
\[
\sup_{B_{\delta/|b|}(y_0)}|v\circ h_m|\leq (\frac{1}{2}+\frac{1}{6})\sup_{B_{\delta/|b|}(y_0)}(u\circ h_m)=\frac{2}{3}\sup_{B_{\delta/|b|}(y_0)}(u\circ h_m)
<\eta_0\sup_{B_{\delta/|b|}(y_0)}(u\circ h_m).
\]
Thus, for $m=1,2$, $|A_{s,h_m,n_0}(f_{\sigma} v)(y)|\leq \eta_0 A_{\sigma,h_m,n_0}(f_{\sigma} u)(y)$
for all $y\in B_{\delta/|b|}(y_0)$. So, Case $h_m$ holds with $y_1=y_0$.
\\[2mm]
{\bf Case 2.} Suppose the reverse; that is, suppose that 
$\inf_{B_{\delta/|b|}(y_0)}|v\circ h_m|> \frac{1}{2} \sup_{B_{\delta/|b|}(y_0)}(u\circ h_m)$ for $m=1,2$. 

For $m=1,2$, write $A_{s,h_m,n_0}(f_{\sigma} v)(y)=r_m(y)e^{i\theta_m(y)}$. Let $\theta(y)=\theta_1(y)-\theta_2(y)$.
Choose $\delta$ as in~\eqref{eq-delta} and recall $\Delta=\frac{2\pi}{D}$.
A calculation~\cite[Lemma 2.3]{BalVal} shows that if $\cos\theta\leq 1/2$ then
$r_1e^{i\theta_1}+ r_2e^{i\theta_2}\leq\max\{\eta_0 r_1+r_2, r_1+\eta_0 r_2\}$.
Thus, the conclusion follows once we show that  $\cos\theta(y)\leq 1/2$, or equivalently $|\theta(y)-\pi|<2\pi/3$,
for all $y\in B_{\delta/|b|}(y_1)$ for some $y_1\in B_{\Delta/|b|}(y_0)$.
In what follows we show that $|\sup_{B_{\delta/|b|}(y_1)}\theta-\pi|<2\pi/3$, for some $y_1\in B_{\Delta/|b|}(y_0)$.

We start by restricting to  $B_{\xi/|b|}(y_0)$,  where $\xi=\delta+\Delta$.
Note that $\theta=V-b\psi$, where $\psi=\psi_{h_1,h_2}$ is the quantity defined in UNI and $V=\arg(v\circ h_1)-\arg(v\circ h_2)$.
We first estimate $\Osc_{B_{\xi/|b|}(y_0)}V$.
For this purpose, we recall a basic trigonometry result (also used in in~\cite{BalVal} and~\cite{AM}): if $|z_1|, |z_2|\geq c$
and $|z_1-z_2|\leq c(2-2\cos\omega)^{1/2}$ for $c>0$ and $|\omega|<\pi$ then $|\arg(z_1)-\arg(z_2)|\leq \omega$.

Since $(u,v)\in\CC_b$ and $\xi<4\pi/D$ for $m=1,2$,
we have by~\eqref{eq-n0-1}
\begin{align}\label{eq-sup1}
\nonumber\Osc_{B_{\xi/|b|}(y_0)}(v\circ h_m)&\leq C_{10} \rho_0^{-n_0}\frac{4\pi}{D}\sup_{B_{\xi/|b|}(y_0)}(u\circ h_m)\\
&\leq \frac{1}{4}(2-2\cos\frac{\pi}{12})^{1/2}\sup_{B_{\xi/|b|}(y_0)}(u\circ h_m).
\end{align}
Recalling the assumption of Case 2,
\begin{align}\label{eq-sup2}
\nonumber\sup_{B_{\xi/|b|}(y_0)}|v\circ h_m|&\geq 
\Big|\sup_{B_{\xi/|b|}(y_0)}|v\circ h_m|-\Osc_{B_{\xi/|b|}(y_0)}(v\circ h_m)\Big|\\
&\geq \frac{1}{2} \sup_{B_{\delta/|b|}(y_0)}(u\circ h_m)-\frac{1}{4}\sup_{B_{\xi/|b|}(y_0)}(u\circ h_m)=\frac{1}{4}\sup_{B_{\xi/|b|}(y_0)}(u\circ h_m).
\end{align}
By equations~\eqref{eq-sup1} and~\eqref{eq-sup2},
\[
\sup_{z_1, z_2 \in B_{\delta/|b|}(y_0)}
\Big|\arg(v\circ h_m(z_1))-\arg(v\circ h_m(z_2))\Big|\leq\frac{\pi}{12},
\]
and thus
\begin{align}\label{eq-oscV}
\Osc_{B_{\xi/|b|}(y_0)}V\leq \frac{\pi}{6}.
\end{align}

Next, recall the UNI assumption in Subsection~\ref{subsec-uni}. Note that for any $z\in B_{\Delta/|b|}(y_0)$,
\[
|b(\psi(z)-\psi(y))|\geq |b||z-y_0|\inf|\psi'|\geq D|b||z-y_0|=\frac{2\pi}{\Delta}|b||z-y_0|.
\]
Since $|b|>2\Delta$, the ball $B_{\Delta/|b|}(y_0)\subset Y$ contains an interval of length at least $\Delta/|b|$. Hence,
as $z$ varies in $B_{\Delta/|b|}(y_0)$,  it fills out an interval around $0$ of length at least $2\pi$
$b(\psi(z)-\psi(y))$. This means that we can choose $y_1\in B_{\Delta/|b|}(y_0)$ such that
\[
b(\psi(y_1)-\psi(y))=\theta(y_0)-\pi\mod 2\pi.
\]
Note that $\theta(y_0)-V(y_0)+b\psi(y_0)=0$. Using the above displayed equation,
\[
\theta(y_1)-\pi=V(y_1) -b\psi(y_1)-\pi+\theta(y_0)-V(y_0)+b\psi(y_0)=V(y_1)-V(y_0).
\]
Together with~\eqref{eq-oscV}, the above equation implies that $|\theta(y_1)-\pi|\leq \pi/6$.
Recalling $\sup_Y|\psi'|\leq C_0$ and our choice of $\delta$,
\begin{align*}
\Big|\sup_{B_{\delta/|b|}(y_1)}\theta-\pi\Big|&\leq \frac{\pi}{6}+\sup_{B_{\delta/|b|}(y_1)}\Big| \theta-\theta(y_1)\Big| \\
&\leq \frac{\pi}{6}+|b| \sup_{B_{\delta/|b|}(y_1)}\Big| \psi-\psi(y_1)\Big|
+\Osc_{B_{\delta/|b|}(y_1)} V+\Osc_{B_{\Delta/|b|}(y_0)}  V \\
& \leq \frac{\pi}{6}+C_0\delta+2\Osc_{B_{\xi/|b|}(y_0)}V \leq\frac{4\pi}{6}=\frac{2\pi}{3},
\end{align*}
which ends the proof.
\end{proof} 

Let $I^p$ be a closed interval contained in  an atom of $\P_k$ such that 
if Lemma~\ref{lemma-cancellation} holds on $B_{\delta/|b|}(y_1)$,
we also have $B_{\delta/|b|}(y_1)\subset I^p$. Write $type(I^p)=h_m$ if we are in case $h_m$. Then we can find finitely many disjoint intervals $I_j^p=[a_j, b_{j+1}]$,
$j=0,\dots, N-1$ (with $0=b_0 \leq a_0<b_1<a_1<\ldots<b_N\leq a_n=1$) of $type(I_j^p) \in \{h_1, h_2\}$ with $\diam(I_j^p)\in [\delta/|b|, 2\delta/|b|]$ and gaps
$J_j^p=[b_j,a_j]$, $j=0,\dots, N$ with $\diam(J_j^p)\in (0, 2\Delta/|b|]$.

Let $\chi:Y\to[\eta, 1]$,  with $\eta\in [\eta_0,1)$ be a $C^1$ function as 
constructed below (as in~\cite{AM, BalVal}):
\begin{itemize}
\item Let $p\in\P_k$, $h\in\H_n$ for $n\in\N$ and write $h|_p: p\to h(p)$. Set $\chi\equiv 1$ on $Y\setminus (h_1(p) \cup h_2(p))$.

\item On $h_1(p)$ we require that $\chi(h_1(y))=\eta$ for all $y$ lying in the middle third of an interval of type $h_1$ and that 
$\chi(h_1(y))=1$ for all $y$ not lying in an interval of type $h_1$.

\item On $h_2(p)$ we require that $\chi(h_2(y))=\eta$ for all $y$ lying in the middle third of an interval of type $h_2$ and that 
$\chi(h_2(y))=1$ for all $y$ not lying in an interval of type $h_2$.
\end{itemize}

Since $\diam (I_j^p)\geq\delta/|b|$, we can choose $\chi$ to be $C^1$ with $|\chi'|\leq \frac{3(1-\eta)|b|}{\delta P}$ where $P=\min_{m=1,2}\{\inf |h'_m|\}$.
From here on we choose $\eta\in [\eta_0,1)$ sufficiently close to $1$ so that $|\chi'|\leq |b|$.

Since $p\in\P_k$ is arbitrary in the statement of Lemma~\ref{lemma-cancellation} and the construction of $\chi$ above,
we obtain

\begin{cor}\label{cor-canc}
Let $\delta,\Delta$ be as in Lemma~\ref{lemma-cancellation}. Let $|b|\geq 4\pi/D$ and $(u,v)\in\CC_b$.
Let $\chi=\chi(b,u,v)$ be the $C^1$ function described above. Then $|\tilde \L_s^{n_0}v(y)|\leq \tilde \L_{\sigma}^{n_0}(\chi u)(y)$,
for all $s=\sigma+ib$, $|\sigma|<\eps$ and all $y\in Y$.
\end{cor}

The following intervals $\hat I^p$ and $\hat J^p$ are constructed as in~\cite{AM, BalVal}.
Let $\hat I^p=\cup_{j=0}^{N-1}\hat I_j^p$, where $\hat I_j^p$ denotes the middle third of $I_j^p$.
Let $\hat J_j$ be the interval consisting of $J_j$ together with the rightmost third of $I_{j-1}^p$
and the leftmost third of $I_j^p$. Define $\hat J_0^p$ and $\hat J_p^N$ with the obvious modifications.
By construction, $\diam(\hat I_j^p)\geq \frac{1}{3}\frac{\delta}{|b|}$ and $\diam(\hat J_j^p)\geq (\frac{4}{3}+2\Delta)\frac{\delta}{|b|}$.
Hence, there is a constant $\delta'=\delta/(4\delta+6\Delta)>0$ (independent of $b$) such that $\diam(\hat I_j^p)\geq \delta'\diam(\hat J_j^p)$
for $j=0,\dots, N-1$. 

\begin{prop}\label{prop-intsupinfw} Suppose that $w$ is a positive function with $\frac{\sup_p w}{\inf_p w}\leq M$ for some $M>0$.
Then $\int_{\hat I^p}w\, d\Leb\geq \delta'' \int_{\hat J^p}w\, d\Leb$, where $\delta''=(2 M)^{-1} \delta'$.
\end{prop}

\begin{proof}Compute that
\begin{eqnarray*}
\int_{\hat I^p}w\, d\Leb &\geq& \Leb(\hat I_j^p)\inf_p w\geq M^{-1} \delta'\Leb(\hat J_j^p) \sup_p w \\
&=& 2\delta''\Leb(\hat J_j^p) \inf_p w \geq  2\delta'' \int_{\hat J_j^p}w\, d\Leb.
\end{eqnarray*}
Here the factor $2$ takes care of the intervals $\hat J_0^p$ and $\hat J_p^N$.~\end{proof}

\section{Invariance of the cone}
\label{sec-invcone}

Recall that the cone $\CC_b$ was defined in \eqref{eq:cone}. 
The main result of this section is:

\begin{lemma}\label{lemma-invcone} Assume $|b| \geq 2$. Then
$\CC_b$ is invariant under 
$(u,v) \mapsto (\tilde\L_{\sigma}^{n_0}(\chi u), \tilde\L_s^{n_0}v)$,
where $\chi = \chi(b,u,v) \in C^1(Y)$ comes from Corollary~\ref{cor-canc}.
\end{lemma}

\begin{proof}
Since $\chi u\geq \eta u>0$ and $\tilde \L_{\sigma}$ is a positive operator we have 
$\tilde\L_{\sigma}^{n_0}(\chi u)>0$. The condition $|\tilde\L_s^{n_0}v|\leq \tilde\L_{\sigma}^{n_0}(\chi u)$
follows from Corollary~\ref{cor-canc}. In what follows we check the other cone conditions for the pair
$(\tilde\L_{\sigma}^{n_0}(\chi u), \tilde\L_s^{n_0}v)$.

For simplicity of exposition, we assume that $n_0 = 2qk$ for some $q \geq 1$.
We will start with invariance of the exponential jump-size and oscillation
conditions under $(u,v) \mapsto (\tilde \L_{\sigma}^n u, \tilde \L_s^n v)$
for a smaller exponent $n = 2k$. Iterating
this, we get to the required exponent $n_0$.
Hence define
\begin{eqnarray*}
(u_1, v_1) &=& (\tilde \L_{\sigma}^{n} u, \tilde \L_s^{n}v) \\
(u_2, v_2) &=& (\tilde \L_{\sigma}^{n} u_1, \tilde \L_s^{n}v_1) \\
\vdots \ \quad  & \vdots & \qquad \quad \vdots \\
(u_{q-1}, v_{q-1}) &=& (\tilde \L_{\sigma}^{n} u_{q-2}, \tilde \L_s^{n}v_{q-2}) \\
(u_q, v_q) &=& (\tilde \L_{\sigma}^{n} u_{q-1}, \tilde \L_s^{n} v_{q-1})
= (\tilde \L_{\sigma}^{n_0} u, \tilde \L_s^{n_0}v).
\end{eqnarray*}
Since $|v| \leq u$, this construction shows that $|v| \leq u$ for all
$1 \leq i \leq q$.
We will now show by induction that
$(u_i, v_i)$ satisfies \eqref{eq:EJ} and
$\Osc_I v_i \leq  C_{10} |b| \Leb(I) \sup_I u_i + C_8 E_I(u_i)$
for all $1 \leq i \leq q$.

{\bf The {\lq}exponential decrease of jump-sizes{\rq} condition in $\CC_b$.}
Without loss of generality we can refine (if needed) 
the partition $\P_k$ such that 
\begin{equation}\label{eq:Pk}
C_{10} |b| \Leb([\xi_{i-1}, \xi_i]) \leq \twothird,
\end{equation}
for all $i$.
Then the oscillation condition applied to $(u,v=u)$
combined with \eqref{eq:Pk} and the fact that $E_I(u) \leq \frac{1}{12} \sup_p u$ give
$\sup_p u - \inf_p u = \Osc_p u \leq (\frac23 + \frac{1}{12}) \sup_p u$.
Therefore
$\frac{\sup u|_p}{\inf u|_p} \leq 4$ for each $p \in \P_k$.
The invariance of the exponential jump-size condition
follows by Proposition~\ref{prop:inductive_jumpsize}, that is:
the pair $(\tilde \L_{\sigma}^n u, \tilde \L_s^n v)$ satisfies
\eqref{eq:EJ} as well.

{\bf The {\lq}oscillation{\rq} condition in $\CC_b$.}
For the invariance of the oscillation condition,
we need to verify
$$
\Osc_I(\tilde\L_s^nv)
\leq C_{10} |b| \Leb(I) \sup_{x \in I}(\tilde\L_{\sigma}^n u)(x)
+ C_8 E_I (\tilde\L_{\sigma}^nu).
$$
For this purpose, we split $\Osc_I(\tilde\L_s^nv)$ into a sum of jump-sizes
at non-onto branches
(\ie $\partial \dom(h) \cap I^\circ \neq \emptyset$, corresponding to the ``created'' discontinuities), 
and a sum of onto branches (which includes ``propagated'' discontinuities).
Because of \eqref{eq:oscsize}, this gives the following:
\begin{eqnarray*}
\Osc_I(\tilde\L_s^nv) &\leq & \sum_{h \in \H_n, \partial \dom(h) \cap I^\circ \neq \emptyset }
\Size \Big( |h'| e^{s \varphi_n \circ h(x)} \frac{(f_{\sigma} v) \circ h }{\lambda_{\sigma}^n f_{\sigma} } \Big)(\partial \dom(h) \cap I^\circ) \\
&& + \sum_{h \in \H_n, \dom(h) \cap I^\circ \neq \emptyset} \Osc_I\Big( |h'| e^{s \varphi_n \circ h} \frac{(f_{\sigma} v) \circ h }{\lambda_{\sigma}^n f_{\sigma} } \Big) 
\\
&=& O_1 + O_2.
\end{eqnarray*}
For the term $O_1$ we use Proposition~\ref{prop:inductive_jumpsize}, and recall that $I \subset p$, so each created discontinuity $x$ in this sum belong to $X'_j$ 
for some $k < j \leq n$. We obtain
\begin{equation}\label{eq:E1}
O_1 \leq C_7 \sum_{j=k+1}^n \rho^{-j} \sum_{x \in X'_j \cap I^\circ}
\tilde\L_{\sigma}^n u(x), 
\end{equation}
which contributes to $E_I(\tilde\L_{\sigma}^n(\chi u))$.
 
Now for the sum $O_2$ (concerning the interiors of $\dom(h)$, $h \in \H_n$),
we decompose the summands into five parts, according to the five factors
$|h'|$, $e^{s \varphi_n \circ h}$, $f_{\sigma} \circ h$, $1/f_{\sigma}$ and $v \circ h$
of which the oscillations have to be estimated.
The estimates for this five parts are as follows.\\
{\bf The term with $|h'|$.}
For each $h \in \H_n$ we have
$1 = h' \circ F^n \cdot(F^n)'$ and $0 = h'' \circ F^n \cdot ((F^n)')^2 +
h' \circ F^n \cdot (F^n)''$.
Using Adler's condition \eqref{eq:adler} for the branches of $F^n$, 
\begin{equation}\label{eq:hpp}
|h''(\xi)| = \frac{ |(F^n)'' \circ h(\xi)|}{|(F^n)' \circ h(\xi)|^2} \cdot |h'(\xi)| \leq C_1 |h'(\xi)|
\end{equation}
for each $n \geq 1$ and $\xi \in a \in \alpha^n$. 
Hence by the Mean Value Theorem,
$$
\Osc_{I^\circ}(|h'|) \leq  \Leb(I) |h''(\xi)|
\leq  C_1 \Leb(I) |h'(\xi)|
\leq  C_1e^{C_1}  \Leb(I) \inf_{x \in \dom(h) \cap I} |h'(x)|.
$$
Summing over all $h \in \H_n$ with $\dom(h) \cap I^\circ \neq \emptyset$, we get
\begin{equation}\label{eq:O1}
\sum_{\stackrel{h \in \H_n}{\dom(h) \cap I^\circ \neq \emptyset}}
\Osc_{I^\circ}(|h'|) \sup_{x \in \dom(h) \cap I^\circ} e^{\sigma \varphi_n \circ h(x)} 
\frac{(f_{\sigma} |v|) \circ h(x)}{\lambda_{\sigma}^n f_{\sigma}(x)}
\leq  C_1 e^{C_1}  \Leb(I)  \sup_{x \in I} (\tilde \L_{\sigma}^n u)(x).
\end{equation}

\noindent
{\bf The term with $e^{s \varphi_n \circ h}$}.
Write 
$\varphi_n(x) = \sum_{i=0}^{m-1} \varphi \circ F^i(x)$ and
$h = h_n\circ h_{n-1} \circ \dots \circ h_1 \in \H_n$ where $h_j \in \H_1$ for $1 \leq j \leq n$. 
Then by \eqref{eq:C2}
\begin{align}\label{eq:phin}
|(\varphi_n \circ h)'| &\leq \sum_{j=0}^{n-1}| (\varphi \circ h_{n-j} \circ F^{j+1} \circ h)'| 
= \sum_{j=0}^{n-1} |(\varphi \circ h_{n-j})'| \cdot |(F^{j+1} \circ h)'| \nonumber \\
& \leq C_2 \sum_{j=0}^{n-1} \rho_0^{-(n-(j+1))} \leq \frac{C_2 \rho_0}{\rho_0-1}
 =: C'_2.
\end{align}
By the Mean Value Theorem 
$\frac{\sup_{x \in I} e^{\sigma \varphi_n \circ h(x)}}
{\inf_{x \in I} e^{\sigma \varphi_n \circ h(x)}}
\leq e^{\sigma (\varphi_n \circ h)'(\xi) \Leb(I)}
\leq e^{\eps C'_2}$.
Therefore
\begin{eqnarray*}
\Osc_{I^\circ}(e^{s \varphi_n \circ h}) &=& |s| e^{\sigma \varphi_n \circ h(\xi)}
|(\varphi_n \circ h)'(\xi)| \Leb(I) \\
&\leq & (1+\eps)|b|  \frac{\sup_{x \in I} e^{\sigma \varphi_n \circ h(x)}}
{\inf_{x \in I} e^{\sigma \varphi_n \circ h(x)}}
\inf_{x \in I} e^{\sigma \varphi_n \circ h(x)} \sup_{x \in I} (\varphi_n \circ h)'(x)\\
&\leq & (1+\eps) e^{\eps C'_2} C'_2 |b| \Leb(I) \inf_{x \in I} 
e^{\sigma \varphi_n \circ h(x)}.
\end{eqnarray*}
Summing over all $h \in \H_n$ with $\dom(h) \cap I^\circ \neq \emptyset$,
this gives
\begin{align}\label{eq:O2}
\sum_{\stackrel{h \in \H_n}{\dom(h) \cap I^\circ \neq \emptyset}}
\Osc_{I^\circ}(e^{s \varphi_n \circ h}) & \sup_{x \in \dom(h) \cap I^\circ} |h'(x)| 
\frac{(f_{\sigma} |v|) \circ h(x)}{\lambda_{\sigma}^n f_{\sigma}(x)} \nonumber \\
& \leq  (1+\eps) e^{\eps C'_2}  C'_2 |b| \Leb(I) \sup_{x \in I} (\tilde \L_{\sigma}^n u)(x).
\end{align}

\noindent
{\bf The term with $f_{\sigma} \circ h$}.
Applying Lemma~\ref{lem:jump_fsigma}, part 2 to $f_{\sigma} \circ h$ we find
\begin{equation}\label{eq:oscfh}
\Osc_{I^\circ}(f_{\sigma} \circ h) \leq C_6 \Leb(h(I)) \inf_{x\in h(I)} f_{\sigma}(x) + C_7 E_{h(I)}(f_{\sigma}).
\end{equation}
For an arbitrary $h \in \H_n$, the first term in \eqref{eq:oscfh}, multiplied by
$\sup_{x \in \dom(h) \cap I^\circ} |h'(x)|\ | e^{s \varphi_n \circ h(x)} | \
\frac{|v| \circ h(x)}{\lambda_{\sigma}^n f_{\sigma}(x)}$ is bounded by
$$
 C_6 \Leb(h(I))\sup_{x \in \dom(h) \cap I^\circ} |h'(x)| 
e^{\sigma \varphi_n \circ h(x)} 
\frac{ (f_{\sigma} u) \circ h(x)}{\lambda_{\sigma}^n f_{\sigma}(x)}.
$$
Summing over all $h \in \H_n$ with $\dom(h) \cap I^\circ \neq \emptyset$
gives
\begin{equation}\label{eq:O3}
\sum_{\stackrel{h \in \H_n}{\dom(h) \cap I^\circ \neq \emptyset}}
 C_6 \Leb(h(I)) \sup_{x \in \dom(h) \cap I^\circ} |h'(x)| 
e^{\sigma \varphi_n \circ h(x)} 
\frac{ (f_{\sigma} u) \circ h(x)}{\lambda_{\sigma}^n f_{\sigma}(x)}
\leq  C_6 \rho_0^{-n} \Leb(I) \sup_{x \in I} (\tilde \L_{\sigma}^n u)(x).
\end{equation} 
The second term in \eqref{eq:oscfh} is a sum over propagated discontinuities $x \in I^\circ$, and for each $x$ we let $\tilde h \in \H_n$ be the inverse branch such that
$f_{\sigma}$ has a discontinuity at $y = \tilde h(x)$, and $j>k$ is such that $x \in X'_j$.
By Lemma~\ref{lem:jump_fsigma}
the term in $E_{h(I)}(f_{\sigma})$ related to $y$ is bounded by
$C_7 \rho^{-3(j-n)} f_{\sigma}(y)$. 
Multiplied by
$|\tilde h'(x)|\ | e^{s \varphi_n \circ\tilde h(x)} |\ \frac{|v| \circ\tilde h(x)}{\lambda_{\sigma}^n f_{\sigma}(x)}$,
and using \eqref{eq:Lnu} to obtain an upper bound for $u\circ \tilde h(x) = u(y)$, this gives
\begin{eqnarray*}
\frac{C_7}{\rho^{3(j-n)}} f_{\sigma}(y)
|\tilde h'(x)|\ e^{\sigma \varphi_n \circ\tilde h(x)} \frac{|v| \circ\tilde h(x)}{\lambda_{\sigma}^n f_{\sigma}(x)}
&\leq & \frac{C_7}{\rho^{3(j-n)}} \rho^{-3n} 
\frac{(f_{\sigma} u) \circ \tilde h(x)}{\lambda_{\sigma}^n f_{\sigma}(x)} \\
&\leq & C_7  \rho^{-j} \frac{\sup f_{\sigma}}{\inf f_{\sigma}}
\frac{\sup u|_p}{\inf u|_p} \frac{1}{\rho^k C_9 \Leb(p)} \tilde \L_{\sigma}^n u(x).
\end{eqnarray*}
Since $\frac{\sup u|_p}{\inf u|_p} \leq 4$,
the bound on $\Leb(p)$ in \eqref{eq:k} gives
$\frac{\sup f_{\sigma}}{\inf f_{\sigma}}
\frac{\sup u|_p}{\inf u|_p} \frac{1}{\rho^k C_9 \Leb(p)} \leq 1$.
Hence, summing over all propagated discontinuities $x \in I^\circ$
and corresponding branches, we get
\begin{equation}\label{eq:E2}
C_7 \sum_{j > n} \sum_{x \in X'_j \cap I^\circ} \rho^{-3(j-n)}
f_{\sigma}(y)
|h'(x)|\ e^{\sigma \varphi_n \circ h(x)} \frac{|v| \circ h(x)}{\lambda_{\sigma}^n f_{\sigma}(x)} \le
C_7  \sum_{j > n} \rho^{-j} \sum_{x \in X'_j \cap I^\circ} \tilde\L_{\sigma}^n u(x).
\end{equation}
which contributes to $E_I(\tilde\L_{\sigma}^n u)$.

\noindent
{\bf The term with $1/f_{\sigma}$}. 
Applying Lemma~\ref{lem:jump_fsigma}, part 2.\ to $f_{\sigma} \circ h$ we find
\begin{equation}\label{eq:osc1f}
\Osc_{I^\circ}(1/f_{\sigma}) \leq C_6 \Leb(I) \inf_{x\in h(I)} 1/f_{\sigma}(x) + C_7 E_I(1/f_{\sigma}).
\end{equation}
For $h \in \H_n$, the first term of \eqref{eq:osc1f}, multiplied by
$\sup_{x \in \dom(h) \cap I^\circ} |h'(x)|\ | e^{s \varphi_n \circ h(x)} | \
\frac{(f_{\sigma}|v|) \circ h(x)}{\lambda_{\sigma}^n}$ is bounded by
$$
 C_6 \Leb(I) \sup_{x \in \dom(h) \cap I^\circ} |h'(x)| 
e^{\sigma \varphi_n \circ h(x)} 
\frac{ (f_{\sigma} u) \circ h(x)}{\lambda_{\sigma}^nf_{\sigma}(x)}.
$$
Summing over all $h \in \H_n$ with $\dom(h) \cap I^\circ \neq \emptyset$
gives
\begin{equation}\label{eq:O4}
\sum_{\stackrel{h \in \H_n}{\dom(h) \cap I^\circ \neq \emptyset}}
 C_6 \Leb(I) \sup_{x \in \dom(h) \cap I^\circ} |h'(x)| 
e^{\sigma \varphi_n \circ h(x)} 
\frac{ (f_{\sigma} u) \circ h(x)}{\lambda_{\sigma}^n f_{\sigma}(x)}
\leq  C_6 \Leb(I) \sup_{x \in I} (\tilde \L_{\sigma}^n u)(x).
\end{equation}
The second term of \eqref{eq:osc1f} is a sum over propagated discontinuities $x \in I^\circ$. Take $j > k$ such that $x \in X'_j$.
Lemma~\ref{lem:jump_fsigma} gives that 
the term in $E_I$ related to $x$ is bounded by
$C_7 \rho^{-3j}/f_{\sigma}(x)$. 
Multiplying with 
$|h'(x)| \ |e^{\sigma \varphi_n \circ h(x)}|\ \frac{(f_{\sigma} u) \circ h(x)}{\lambda_{\sigma}^n}$
and then summing over all $x \in \cup_{j > k} X'_j \cap I^\circ$
and $h \in \H_n$ with $x \in \dom(h)$ gives
\begin{equation}\label{eq:E3}
C_7  \sum_{j > k} \rho^{-3j} \sum_{x \in X'_j \cap I^\circ}
|h'(x)| e^{\sigma \varphi_n \circ h(x)} \frac{(f_{\sigma} u) \circ h(x)}{\lambda_{\sigma}^n f_{\sigma}(x) }
\leq C_7  \sum_{j > k} \rho^{-j}\sum_{x \in X'_j \cap I^\circ}
(\tilde \L_{\sigma}^n u)(x),
\end{equation}
which contributes to $E_I(\tilde \L_{\sigma}^n u)$.

\noindent
{\bf The term with $v$}.
Using the cone condition for $v$, we obtain
\begin{eqnarray}\label{eq:oscvh}
\Osc_{I^\circ}(v \circ h) &\leq & C_{10} \Leb(h(I)) |b| \sup_{x \in h(I)} u(x) 
+ C_8 E_{h(I)}(u) \nonumber \\
&\leq & \rho_0^{-n} \frac{\sup u|_{h(I)}}{\inf u|_{h(I)}} C_{10} \Leb(I)\ |b|\ \inf_{x \in h(I)} u(x) 
+ C_8 E_{h(I)}(u).
\end{eqnarray}
For $h \in \H_n$, the first term of \eqref{eq:oscvh}, 
multiplied by
$\sup_{x \in \dom(h) \cap I^\circ} |h'(x)|\ | e^{s \varphi_n \circ h(x)} | \
\frac{f_{\sigma} \circ h(x)}{\lambda_{\sigma}^n f_{\sigma}(x)}$, is bounded by
$$
4\rho_0^{-n}  C_{10} |b| \Leb(I) \sup_{x \in \dom(h) \cap I^\circ} |h'(x)| 
e^{\sigma \varphi_n \circ h(x)} 
\frac{ (f_{\sigma} u) \circ h(x)}{\lambda_{\sigma}^nf_{\sigma}(x)}.
$$
Summing over all $h \in \H_n$ with $\dom(h) \cap I^\circ \neq \emptyset$
gives
\begin{equation}\label{eq:O5}
\sum_{\stackrel{h \in \H_n}{\dom(h) \cap I^\circ \neq \emptyset}}
\frac{4C_{10}}{\rho_0^n} |b| \Leb(I) \sup_{x \in \dom(h) \cap I^\circ} |h'(x)| 
e^{\sigma \varphi_n \circ h(x)} 
\frac{ (f_{\sigma} u) \circ h(x)}{\lambda_{\sigma}^nf_{\sigma}(x)}
\leq  \frac{4C_{10}}{\rho_0^n} |b| \Leb(I) \sup_{x \in I} (\tilde \L_{\sigma}^n u)(x).
\end{equation}
The second term of \eqref{eq:oscvh} is a sum over propagated discontinuities $x \in I^\circ$. For each such $x$ we let $\tilde h \in \H_n$ be the inverse branch such that 
$v$ has a discontinuity at $y = \tilde h(x)$, and $j$ is such that $x \in X'_j$.
\\
{\bf Case a:} Assume that $j-n > k$.
Since $u$ has exponentially decreasing jump-sizes, we 
get that the term in $E_{h(I)}$ related to $y$ is bounded by
$C_7 \rho^{-(j-n)} u(y)$. 
After multiplying by
$|\tilde h'(x)|\ | e^{s \varphi_n \circ\tilde h(x)} |\ \frac{f_{\sigma} \circ\tilde h(x)}{\lambda_{\sigma}^n f_{\sigma}(x)}$,
and using \eqref{eq:Lnu} for an upper bound of
$u \circ \tilde h(x) = u(y)$, we have
\begin{eqnarray*}
C_7 \rho^{-(j-n)} u(y)
|h'(x)|\ e^{\sigma \varphi_n \circ h(x)} \frac{f_{\sigma} \circ h(x)}{\lambda_{\sigma}^n f_{\sigma}(x)}
&\leq & C_7 \rho^{-(j-n)} \rho^{-3n} \frac{(f_{\sigma} u) \circ \tilde h(x)}{f_{\sigma}(x)} \\
&\leq & C_7  \rho^{-j} 
\frac{\sup f_{\sigma}}{\inf f_{\sigma}}
\frac{\sup u|_p}{\inf u|_p} \frac{1}{\rho^k C_9 \Leb(p)} \tilde \L_{\sigma}^n u(x) \\
&\leq & \frac{C_7}{C_8}  \rho^{-j} \tilde \L_{\sigma}^n u(x),
\end{eqnarray*}
because $\frac{\sup u|_p}{\inf u|_p} \leq 4$,
and using the bound on $\Leb(p)$ from \eqref{eq:k}.
\\
{\bf Case b:} Assume that $j-n \leq k$. Then \eqref{eq:jumpsize}
doesn't apply to 
the term in $E_{h(I)}$ related to $y$, so it can only be bounded by
$u(y)$. 
Multiplied by
$|\tilde h'(x)|\ | e^{s \varphi_n \circ\tilde h(x)} |\ \frac{f_{\sigma} \circ\tilde h(x)}{\lambda_{\sigma}^n f_{\sigma}(x)}$, and
 using \eqref{eq:Lnu} for obtaining an upper bound of
$u \circ \tilde h(x) = u(y)$, we have
\begin{eqnarray*}
u(y) |h'(x)|\ e^{\sigma \varphi_n \circ h(x)} \frac{f_{\sigma} \circ h(x)}{\lambda_{\sigma}^n f_{\sigma}(x)}
&\leq & \rho^{-3n} \frac{(f_{\sigma} u) \circ \tilde h(x)}{f_{\sigma}(x)} \\
&\leq & \rho^{-2(n-k)} \frac{\sup f_{\sigma}}{\inf f_{\sigma}}
\frac{\sup u|_p}{\inf u|_p} \frac{1}{\rho^k C_9 \Leb(p)} \rho^{-j} \tilde \L_{\sigma}^n u(x) \\
&\leq &  \frac{1}{C_8}  \rho^{-j} \tilde \L_{\sigma}^n u(x),
\end{eqnarray*}
because $\frac{\sup u|_p}{\inf u|_p} \leq 4$,
and using the bound on $\Leb(p)$ from \eqref{eq:k}.
Hence, summing over all propagated discontinuities $x \in I^\circ$
and corresponding branches, we get
\begin{equation}\label{eq:E4}
C_7 \sum_{j > n} \sum_{x \in X'_j \cap I^\circ} \rho^{-(j-n)}
f_{\sigma}(y)
|h'(x)|\ e^{\sigma \varphi_n \circ h(x)} \frac{|v| \circ h(x)}{\lambda_{\sigma}^n f_{\sigma}(x)} \le
\frac{C_7}{C_8} \sum_{j > n} \rho^{-j}  \sum_{x \in X'_j \cap I^\circ} \tilde\L_{\sigma}^n u(x),
\end{equation}
which contributes to $E_I(\tilde\L_{\sigma}^n u)$.
This completes the treatment of the five terms.

Combining terms \eqref{eq:O1}, \eqref{eq:O2}, \eqref{eq:O3}, \eqref{eq:O4} and 
\eqref{eq:O5}, the oscillation part is bounded by
$$
\left(C_1 e^{C_1} + (1+\eps)|b| e^{\eps C'_2} C'_2 + (1+\rho_0^{-n})C_6 + 4C_{10}\rho_0^{-n}  \right)
\Leb(I) \sup_I(\tilde \L_{\sigma}^n u)
$$
and by the choice of $C_{10}$ in Subsection~\ref{subsec-uni}, this is less
than $C_{10} |b| \Leb(I) \eta_0 \sup_I(\L_{\sigma}^n u)$
whenever $|b| \geq 2$.

Recall $C_8 = 3C_7/\eta_0$. Combining \eqref{eq:E1}, \eqref{eq:E2}, \eqref{eq:E3} and 
\eqref{eq:E4}, the jump part is bounded by
$$
3C_7 E_I(\tilde\L^n u) \leq C_8 \eta_0 E_I(\tilde\L^n u).
$$
This concludes the induction step,
proving that
\begin{eqnarray*}
\Osc_{I^\circ}(\tilde\L_s^{n_0}v) &\leq & C_{10} \eta_0 |b| \Leb(I) \sup_I
(\tilde\L_{\sigma}^{n_0} u) + C_8 \eta_0 E_I (\tilde\L_{\sigma}^{n_0} u) \\
 &\leq & C_{10} |b| \Leb(I) \sup_I
(\tilde\L_{\sigma}^{n_0} (\chi u)) + C_8 E_I (\tilde\L_{\sigma}^{n_0} (\chi u))
\end{eqnarray*}
 as required.~\end{proof}

\section{Proof of Theorem~\ref{th-main}}
\label{sec-proofmain}

Given Lemma~\ref{lemma-cancellation} and Lemma~\ref{lemma-invcone}, the proof of the $L^2$ contraction for functions in $\CC_b$ goes almost
word by word as the proof of~\cite[Theorem 2.16]{AM} with some obvious modifications. We sketch the argument in Subsection~\ref{sub-L2}.
In Subsection~\ref{sub-L2BV} we deal with arbitrary \BV\ observables  satisfying  a mild condition via the $\|\,\|_b$ norm.
In Subsection~\ref{sub-Completing the argument}, we complete the argument required for the proof of Theorem~\ref{th-main}.

\subsection{$L^2$ contraction for functions in $\CC_b$}
\label{sub-L2}

\begin{lemma}\label{lemma-L2}
There exist $\eps\in (0,1)$ and $\beta\in (0,1)$ such that for all $m\geq 1$, $s=\sigma+ib$, $|\sigma|<\eps$, $|b|\geq\max\{4\pi/D,2\}$,
\[
\int |\tilde \L_s^{mn_0}v|^2\, d\Leb\leq \beta^m \|v\|_\infty^2,
\]
for all $v\in \BV$ such that $(u,v)$ for $u=cst$ satisfy condition~\eqref{eq:EJ} in Definition~\ref{def:expjumpsize}.
\end{lemma}

\begin{proof} Set $u_0\equiv \|v\|_\infty$, $v_0=v$ and for $m\geq 0$, define
\[
u_{m+1}=\tilde \L_{\sigma}^{n_0}(\chi_m u_m),\quad v_{m+1}=\tilde \L_s(v_m),
\]
where $\chi_m$ is a function depending on $b,u_m, v_m$.  Since by definition $(u_0,v_0)\in\CC_b$, it follows from Lemma~\ref{lemma-invcone}
that $(u_m,v_m)\in\CC_b$, for all $m$. Thus, we can construct $\chi_m:=\chi(b,u_m, v_m)$ inductively as in Corollary~\ref{cor-canc}.

As in~\cite{AM, BalVal}, it is enough to show that there exists $\beta\in (0,1)$ such that $\int u_{m+1}^2 \, d\Leb\leq \beta\int u_{m}^2 \, d\Leb$
 for all $m \geq 0$. 
Then $|\tilde \L_s^{mn_0}v|= |\tilde \L_s^{mn_0}v_0|= |v_m|\leq u_m$ and thus,
\[
\int |\tilde \L_s^{mn_0}v|^2\, d\Leb\leq\int u_{m}^2 d\Leb\leq \beta^m 
\int u_{0}^2\, d\Leb=\beta^m \|v\|_\infty^2,
\]
as required. 

Let $\hat I^p,\hat J^p$ be as constructed before the statement of Proposition~\ref{prop-intsupinfw} and note that $Y=(\cup_p\hat I^p)\cup (\cup_p\hat J^p)$.
 Proceeding as in the proof of ~\cite[Lemma 2.13]{AM} (which relies on the use of the Cauchy-Schwartz inequality), we obtain that there exists $\eta_1<1$
such that for any $p\in\P_k$,
$$
u_{m+1}^2(y) \leq
\begin{cases}
  \xi(\sigma)\eta_1 (\tilde \L_0^{n_0} u_{m}^2)(y)& \text{ if } y\in\hat I^p,\\
 \xi(\sigma)(\tilde \L_0^{n_0} u_{m}^2)(y)&  \text{ if } y\in\hat J^p,
\end{cases}
$$
where 
$\xi(\sigma)=\lambda_{\sigma}^{-2n_0}\sup_p(f_0/f_{\sigma})\sup_p(f_{2\sigma}/f_{\sigma})\sup_p(f_{\sigma}/f_0)\sup_p(f_{\sigma}/f_{2\sigma})$.

Since $(u_m,v_m)\in\CC_b$, we have, in particular, that for any $p\in\P_k$, $\sup_p u_m-\inf_p u_m \leq \Osc_p u \leq (\frac{2}{3}+\frac{1}{12})\sup_p u_m$
 and thus, $\frac{\sup_p u_m}{\inf_p u_m}\leq 4$. Similarly, $\frac{\sup_p u_m^2}{\inf_p u_m^2}\leq 16$. Hence,
\begin{align*}
\frac{\sup_p \tilde \L_0^{n_0}(u_m^2)}{\inf_p \tilde \L_0^{n_0}(u_m^2)}
&= \frac{\sup_p \sum_{h\in\H_{n_0}}|h'| (f_0\circ h) (u_m^2\circ h)/f_0}{\inf_p \sum_{h\in\H_{n_0}}|h'| (f_0\circ h) (u_m^2\circ h)/f_0} \\
&\leq 16 \Big(\frac{\sup_p f_0}{\inf_p f_0}\Big)^2\ \frac{\sup_p \sum_{h\in\H_{n_0}}|h'|}{\inf_p \sum_{h\in\H_{n_0}}|h'|}<\infty.
\end{align*}
Let $w:= \tilde \L(u_m^2)$, set $M:=16 \Big(\frac{\sup_p f_0}{\inf_p f_0}\Big)^2\ \frac{\sup_p \sum_{h\in\H_{n_0}}|h'|}{\inf_p \sum_{h\in\H_{n_0}}|h'|}$
 and note that $w$ satisfies the conditions of Proposition~\ref{prop-intsupinfw} for such $M$. For any $p\in\P_k$,  it follows that 
$\int_{\hat I^p} w\, d\Leb\geq \delta''\int_{\hat J^p} w\, d\Leb$ and thus,
\[
\int_{\cup_p \hat I^p} w\, d\Leb\geq \delta''\int_{\cup_p\hat J^p} w\, d\Leb.
\]
From here on the argument goes word by word as the argument used at the end of the proof of~\cite[Theorem 2.16]{AM}.
We provide it here for completeness. Let $\beta'=\frac{1+\eta_1\delta''}{1+\delta''}<1$. Then $\delta''=\frac{1-\beta'}{\beta'-\eta_1}$
and thus, $(\beta'-\eta_1)\int_{\cup_p \hat I^p} w\, d\Leb\geq (1-\beta')\int_{\cup_p\hat J^p} w\, d\Leb$. Since also $Y=(\cup_p\hat I^p)\cup (\cup_p\hat J^p)$,
we obtain
$\eta_1\int_{\cup_p \hat I^p} w\, d\Leb +\int_{\cup_p\hat J^p} w\, d\Leb\leq \beta' \int_{Y} w\, d\Leb$. Putting the above together,
\begin{align*}
\int_{Y} u_{m+1}^2\, d\Leb &\leq  \xi(\sigma)\Big(\eta_1\int_{\cup_p \hat I^p} w\, d\Leb +\int_{\cup_p\hat J^p} w\, d\Leb\Big) \\
&\leq \xi(\sigma)\beta'\int_Y \tilde \L_0^{n_0} (u_{m+1}^2)\, d\Leb=\xi(\sigma)\beta'\int_Y  u_{m}^2\, d\Leb.
\end{align*}
To conclude, recall that by Remark~\ref{rem:lambda}, if necessary, we can shrink $\eps$ such that $\beta:=\xi(\sigma)\beta'<1$ for all $|\sigma|<\eps$.~\end{proof}

\subsection{Dealing with arbitrary \BV\ observables via the $\|\,\|_b$ norm}
\label{sub-L2BV}

The cone $\CC_b$ represents only a specific class of \BV\ observables, namely with 
discontinuities of prescribed size and location.
It is, in fact, the smallest Banach space that is invariant under 
$(u,v) \mapsto (\tilde \L_{\sigma} u, \tilde \L_s v)$
and contains all continuous \BV\ functions.

In this section we are concerned with the behaviour of $\tilde \L_s^r$ acting on \BV\ functions satisfying 
a certain mild condition  (less restrictive than belonging to $\CC_b$). To phrase such a condition we let $C_{11}$ be a positive constant such that
\begin{equation}
\label{eq-cst11}
C_{11} = 64 (1+c)^2 \Big( \frac{\sup f_{\sigma}}{\inf f_{\sigma}} \Big)^2 \,
\frac{\sup f_{2\sigma}}{\inf f_{2\sigma}} 
\left( \frac{\sup f_{\sigma}}{\inf f_{\sigma}}\frac{\sup f_0}{\inf f_0} \right)^2,
\end{equation}
where $c$ is the constant in the statement of Proposition~\ref{prop-LYineq}.
We use the following hypothesis:
\begin{equation}\label{eq:H}
\begin{cases} 
\Var_Y v\leq C_{11} |b|^2 \rho^{m n_0} \|v\|_1 & \text{ if } \sigma \geq 0,\\
\Var_Y(e^{\sigma \varphi_{m n_0}} v) \leq C_{11} |b|^2 \rho^{m n_0} \|e^{\sigma \varphi_{m n_0}}v\|_1 & \text{ if } \sigma < 0.
\end{cases}
\tag{$H_{\sigma,m}$}
 \end{equation}
The next result, Proposition~\ref{prop:to-cone}, says that for $v \in \BV(Y)$ such that if \eqref{eq:H},
then $\tilde \L_s^rv$ is exponentially close to the cone $\CC_b$ in
$\| \ \|_\infty$,
because jumps-sizes of discontinuities of $v$ outside $X_\infty$ die out
at an exponential rate and are not newly created by the dynamics of $F$.

\begin{prop}\label{prop:to-cone}
There exists  $\eps\in (0,1)$ such that
for all $s=\sigma+ib$, $|\sigma|<\eps$, $|b|\geq\max\{4\pi/D,2\}$,
and all $v \in \BV(Y)$ such that \eqref{eq:H}
holds for some $m \geq 1$, there exists
a pair $(u_{m n_0}, w_{m n_0}) \in \CC_b$ such that 
\begin{equation*}
\| \tilde \L_s^{m n_0}v - w_{m n_0} \|_\infty \leq 2C_{10} \ \rho^{-m n_0} |b|\| v \|_\infty
\ \mbox{ and }\ \| w_{m n_0}\|_\infty \leq\| v \|_\infty.
\end{equation*}
\end{prop}

The above result will allow us to prove

\begin{lemma}\label{lemma-L2BV} 
There exist  $\eps\in (0,1)$ and $\beta\in (0,1)$ such that 
for all $s=\sigma+ib$, $|\sigma|<\eps$, $|b|\geq\max\{4\pi/D,1\}$ and for all 
$m\geq 1$,
\[
\|\tilde \L_s^{3mn_0}v\|_b\leq (1+|b|)^{-1} \Var_Y(\tilde \L_s^{3mn_0}v)+(2C_{10}\rho^{-m n_0}|b|+\beta^{m})\|v\|_\infty.
\]
for all $v\in \BV(Y)$ satisfying \eqref{eq:H}.
\end{lemma}

\begin{proof}[Proof of Proposition~\ref{prop:to-cone}]
Let $v \in \BV(Y)$ be arbitrary and take $r = m n_0$ 
(this is a multiple of $k$ because $n_0$ is).
Write $g_r = \tilde \L_s^rv$ and $\bar g_r = \tilde\L_{\sigma}^r |v|$; for every fixed $b\in\R$,
they belong to $\BV(Y)$ as well by Proposition~\ref{prop-LYineq}. 
Therefore $g_r$ has at most countably many discontinuity points, which we denote by $\{ x_i \}_{i \in \N}$. 
Assume throughout this proof that $g_r$ is continuous from the right;
this can be achieved by adjusting $g_r$ at $\{ x_i \}_{i \in \N}$, so it has no effect on the $L^p$-norm for any $p \in [1,\infty]$.

To estimate the jump-size $|a_i|$ of $g_r$ at $x_i \in X'_j$
for some $j \leq r$, 
we note that this discontinuity is created
by non-onto branches of $F^r$, and there exist $y \in X'_1$
and an inverse branch $\tilde h \in \H_{j-1}$ such that
 $y_i = \tilde h(x_i)$. The jump-size of $\tilde \L_s^rv$ at $x_i$
can be expressed as a sum of $h \in \H_{r-(j-1)}$ which in the summand
is composed with $\tilde h$. Then
\begin{align}\label{eq:sizeg}
\Size& \tilde \L_s^r v(x_i) \leq
\sum_{h \in \H_{r-(j-1)}} \!\!\!\!\!\! |(h \circ \tilde h)'(x_i)|\
|e^{s \varphi_{r-(j-1)} \circ h \circ \tilde h(x_i) + s \varphi_{j-1} \circ \tilde h(x_i)}|\
\frac{(f_{\sigma} v) \circ h \circ \tilde h(x_i)}{\lambda_{\sigma}^r f_{\sigma}(x_i)} \nonumber \\
&= \sum_{h \in \H_{r-(j-1)}} \!\!\!\!\!\! |h'(y_i)|\ e^{\sigma \varphi_{r-(j-1)} \circ h(y_i)} 
\frac{(f_{\sigma} v) \circ h(y_i)}{\lambda_{\sigma}^{r-(j-1)} f_{\sigma}(y_i)}  
|\tilde h'(x_i)| \ e^{\sigma \varphi_{j-1} \circ \tilde h(x_i)} 
\frac{f_{\sigma}(y_i)}{\lambda_{\sigma}^{j-1}f_{\sigma}(x_i)} \nonumber \\
&\leq \Big( \sum_{h \in \H_{n-(j-1)}}\!\!\!\!\!\! |h'(y_i)|\
e^{\sigma \varphi_{r-(j-1)} \circ h(y_i)} 
\frac{f_{\sigma} \circ h(y_i)}{\lambda_{\sigma}^{r-(j-1)} f_{\sigma}(y_i)}\Big) \  \| v \|_\infty\
\rho^{-3(j-1)}\ \frac{\sup f_{\sigma}}{\inf f_{\sigma}} \nonumber \\
&\leq \| v \|_\infty\ \rho^3\ \frac{\sup f_{\sigma}}{\inf f_{\sigma}} \rho^{-3j}.
\end{align}
where the sum in brackets in the penultimate line is $1$ because $f_{\sigma}$ 
is an eigenfunction of $\L_{\sigma}$.

For $r>k$, let $Q_{r}$ be an interval partition of $Y$ refining $\P_r$ such that
$\frac12 \rho^{-r} < \Leb(I_{r}) < 2 \rho^{-r}$ for every $I_{r} \in Q_{r}$. 
In fact, by adjusting
$Q_r$ by an arbitrary small amount if necessary, 
we can assume that
$g_r$ and $\bar g_r$ are continuous at every point in $\partial I_r \setminus X_r$, $I_r \in Q_r$.
Construct $w_r$ and $u_r$ to be affine on each $(p,q) = I_r \in Q_r$ such that
$$
\lim_{x \downarrow p} w_r(x) = \lim_{x \downarrow p} g_r(x)
\quad \text{ and } \quad
\lim_{x \uparrow q} w_r(x) = \lim_{x \uparrow q} g_r(x)
$$
and similarly
$$
\lim_{x \downarrow p} u_r(x) = \lim_{x \downarrow p} \bar g_r(x)
\quad \text{ and } \quad
\lim_{x \uparrow q} u_r(x) = \lim_{x \uparrow q} \bar g_r(x).
$$
Then $w_r$ and $u_r$ are continuous on $Y \setminus X_r$ and 
as $\bar g_r \geq |g_r|$, it is immediate that
$u_r \geq |w_r|$ on $Y$.
The main estimate now concerns the oscillation
$$
\Osc_{I_r} g_r = \Osc_{I_r} \left( \sum_{h \in \H_r, I_r \subset \dom(h)}
\frac{e^{s\varphi_r \circ h} |h'|}{\lambda_{\sigma}^r f_{\sigma}} (f_{\sigma} v) \circ h \right)
\quad \text{ for } I_r \in Q_r,
$$
which we will split into five terms similar to the proof of the invariance
of the cone.\\[2mm]
{\bf The term with $|h'|$} is bounded above by
$C_1 e^{C_1} \Leb(I_r) \sup_{x \in I_r} \tilde \L_{\sigma}^r|v|$
as in \eqref{eq:O1}.\\[2mm]
{\bf The term with $e^{s \varphi_n \circ h}$} is bounded above by
$(1+|\sigma|) e^{\sigma C'_2} C'_2 |b| \Leb(I_r) \sup_{x \in I_r} \tilde \L_{\sigma}^r|v|$
as in \eqref{eq:O2}.\\[2mm]
{\bf The term with $1/f_{\sigma}$} is bounded above, by combining
\eqref{eq:O4} and \eqref{eq:E3}, by
$$
C_6 \Leb(I_r) \sup_{x \in I_r} \tilde \L_{\sigma}^r |v| +
C_7 \Leb(I_r) \sum_{j > r} \rho^{-j} \sum_{x \in X'_j \cap I_r} \tilde \L_{\sigma}^r|v|(x).
$$
Here the second term is bounded by
$C_7  N_1 \frac{\rho^{-r}}{\rho-1} \sup_{x \in I_r} \tilde \L_{\sigma}^r|v|
\leq 2C_7\frac{N_1}{\rho-1} \Leb(I_r) \sup_{x \in I_r} \tilde \L_{\sigma}^r|v|$,
where we recall that $\# X'_j \leq N_1$ for all $j \geq 1$.
\\[2mm]
{\bf The term with $f_{\sigma} \circ h$} is bounded above, by combining
\eqref{eq:O3} and \eqref{eq:E2} and arguing as in the
previous case, by
$$
C_6 \rho_0^{-r} \Leb(I_r) \sup_{x \in I_r} \tilde\L_{\sigma}^r|v| + C_7 \sum_{j > r}
\rho^{-j}\!\!\!  \sum_{x \in X'_j \cap I_r} \tilde \L_{\sigma}^r|v|(x)
\leq
(C_6 \rho_0^{-r} + 2C_7 \frac{N}{\rho-1}) \Leb(I_r) \sup_{x \in I_r} \tilde \L_{\sigma}^r|v|.
$$
{\bf The term with $v \circ h$}:
First we treat the case $\sigma \geq 0$. 
By Lemma~\ref{lem:Varvar} (which also gives
a lower bound $r_0$ for $r$)
\[
\|v\|_1 \leq \frac{K_1}{\Leb(I_r)} \int_{F^{-r}(I_r)} |v|\, d\Leb \quad
\text{ for all }\, I_r \in Q_r,
\] 
where $K_1 = 6e^{C_1}/\eta$. Recall that \eqref{eq:H} holds with $C_{11}>1$ as defined in~\eqref{eq-cst11}. Compute that
\begin{align*}
\sum_{\stackrel{h \in \H_r}{I_r \subset \dom(h)}} & \left(
\sup_{x \in I_r} \frac{|e^{s \varphi_r \circ h}| \, |h'|}{\lambda_{\sigma}^r f_{\sigma}} f_{\sigma} \circ h \right) \Osc_{I_r}(v \circ h) 
\leq \rho^{-3r} \frac{\sup f_{\sigma}}{\inf f_{\sigma}} 
\sum_{\stackrel{h \in \H_r}{I_r \subset\dom(h)}} \Osc_{h(I_r)} v\\
\leq\ & \rho^{-3r} \frac{\sup f_{\sigma}}{\inf f_{\sigma}} \Var_{F^{-r}(I_r)} v\leq 2\rho^{-2r} \Leb(I_r) \frac{\sup f_{\sigma}}{\inf f_{\sigma}} \Var_{Y} v\\
\leq\ & 2\rho^{-2r} \Leb(I_r) \frac{\sup f_{\sigma}}{\inf f_{\sigma}}\,
C_{11} |b|^2 \rho^r  \int_Y |v|\, d\Leb \\
\leq\ & 2 C_{11} |b|^2 K_1\rho^{-r} \frac{\sup f_{\sigma}}{\inf f_{\sigma}}\int_{F^{-r}(I_r)} |v| \, d\Leb\\
\leq\ & 2C_{11}|b|^2 K_1\rho^{-r}\left(\frac{\sup f_{\sigma}}{\inf f_{\sigma}}\right)^2 
\sum_{\stackrel{h \in \H_r}{I_r \subset\dom(h)}}
\int_{I_r}  \frac{ |h'|}{f_{\sigma}} (f_{\sigma} |v|) \circ h \, d\Leb.
\end{align*}
Because $\sigma \geq 0$, we can continue as
\begin{align*}
\sum_{\stackrel{h \in \H_r}{I_r \subset\dom(h)}} &
\Big( \sup_{x \in I_r} \frac{|e^{s \varphi_r \circ h}| \, |h'|}{\lambda_{\sigma}^r f_{\sigma}}  f_{\sigma} \circ h \Big) \Osc_{I_r}(v \circ h) \\
\leq\ &
2 C_{11} |b|^2 K_1\rho^{-r} \lambda_{\sigma}^r 
\left(\frac{\sup f_{\sigma}}{\inf f_{\sigma}}\right)^2 
\sum_{\stackrel{h \in \H_r}{I_r \subset\dom(h)}}
\int_{I_r} 
\frac{e^{\sigma \varphi_r \circ h} |h'|}{\lambda_{\sigma}^r f_{\sigma}} 
(f_{\sigma} |v|) \circ h \, d\Leb\\
\leq\ & 2 C_{11}|b|^2 K_1\rho^{-r} \lambda_{\sigma}^r \left(\frac{\sup f_{\sigma}}{\inf f_{\sigma}}\right)^2 
\Leb(I_r) \sup_{x \in I_r} \tilde \L_{\sigma}^r|v|.
\end{align*}
Since $\rho > \lambda_{\sigma}$, we obtain 
the upper bound $\Leb(I_r) \sup_{x \in I_r} \tilde \L_{\sigma}^r|v|$
by taking $r$ sufficiently large.

Now we treat the case $\sigma < 0$. 
By Lemma~\ref{lem:Varvar}
applied to $e^{\sigma \varphi_r} v$ (and with the same lower bound $r_0$ for $r$ as before)
\[
\| e^{\sigma \varphi_r} v\|_1 \leq \frac{K_1}{\Leb(I_r)} \int_{F^{-r}(I_r)} |e^{\sigma \varphi_r} v|\, d\Leb \quad
\text{ for all }\, I_r \in Q_r.
\] 
Note that 
\begin{align*}
\sum_{\stackrel{h \in \H_r}{I_r \subset\dom(h)}} & 
\left(\sup_{x \in I_r} \frac{|e^{s \varphi_r \circ h}| \, |h'|}{\lambda_{\sigma}^r f_{\sigma}} f_{\sigma} \circ h \right) \Osc_{I_r}(v \circ h)  \\
\leq \ & e^{\eps C'_2} 
\sum_{\stackrel{h \in \H_r}{I_r \subset\dom(h)}} 
\left( \sup_{x \in I_r} \frac{|h'|}{\lambda_{\sigma}^r f_{\sigma}} f_{\sigma} \circ h \right) \Osc_{I_r}( (e^{\sigma \varphi_r} v) \circ h) \\
\leq \ & e^{\eps C'_2} \lambda_{\sigma}^{-r} \frac{\sup f_{\sigma}}{\inf f_{\sigma}}
\rho_0^{-r}\ \Osc_{I_r}( (e^{\sigma \varphi_r} v) \circ h).
\end{align*}
Estimating the oscillation as in the case $\sigma \geq 0$, and using
 \eqref{eq:H}, we 
find the upper bound
\begin{align*}
\sum_{\stackrel{h \in \H_r}{I_r \subset\dom(h)}} & 
\left(\sup_{x \in I_r} \frac{|e^{s \varphi_r \circ h}| \, |h'|}{\lambda_{\sigma}^r f_{\sigma}} f_{\sigma} \circ h \right) \Osc_{I_r}(v \circ h)  \\
\leq\ & 2e^{\eps C'_2} C_{11}|b|^2 K_1\rho^{-3r}
\left(\frac{\sup f_{\sigma}}{\inf f_{\sigma}}\right)^2 
\sum_{\stackrel{h \in \H_r}{I_r \subset\dom(h)}}
\int_{I_r}  \frac{ e^{\sigma \varphi_r \circ h} |h'|}{\lambda_{\sigma}^r f_{\sigma}} (f_{\sigma} |v|) \circ h \, d\Leb \\
\leq\ & 2e^{\eps C'_2} C_{11}|b|^2 K_1\rho^{-3r}
\left(\frac{\sup f_{\sigma}}{\inf f_{\sigma}}\right)^2 \Leb(I_r)
\sup_{x \in I_r} \tilde \L_{\sigma}^r |v|.
\end{align*}
By taking $r$ sufficiently large, we obtain again 
 the upper bound $\Leb(I_r) \sup_{x \in I_r} \tilde \L_{\sigma}^r|v|$,
and this finishes the case $\sigma < 0$.

Putting all terms together,
\begin{equation}\label{eq:oscg}
\Osc_{I_r} g_r \leq C_{10} |b| \ \Leb(I_r) \sup_{I_r} \tilde \L_{\sigma}^r|v|,
\end{equation}
and since $w_r$ is an affine interpolation of $g_r$, with the same limit values
at all points $x_i \in X_r$,
\[
\| g_r - w_r\|_\infty \leq
 C_{10} |b| \ \Leb(I_r) \sup_{I_r} \tilde \L_{\sigma}^r|v| \leq
2 C_{10} |b| \rho^{-r} \| v \|_\infty.
\]
 Also, since $w_r$ is an affine interpolation of $g_r$,
we have $\| w_r \| \leq \| g_r \|_\infty \leq \| v \|_\infty$.

We still need to complete the argument why $(u_r, w_r) \in \CC_b$. By \eqref{eq:oscg},  the affine function 
$w_r|_{I_r}$ has slope 
$C_{10} |b| \sup_{I_r} \tilde\L_{\sigma}^r|v| = C_{10} |b| \sup_{I_r} |u_r|$.
This means that for every subinterval $I \subset I_r$,
we also have
$$
\Osc_I w_r \leq C_{10} |b| \Leb(I) \sup_I u_r.
$$
If on the other hand, $I$ intersects several contiguous $I_r \in Q_r$
(but is contained in an atom of $\P_k$),
then we have to include the jump-sizes of discontinuity
points at $\partial I_r$ as well. But since $Q_r$ refines $\P_r$
and $g_q$ is continuous at all boundary points $q \in \partial I_r \setminus X_r$, and the jump-sizes of $g_r$ and $w_r$
coincide at every $x_i \in X'_j$ (and decrease exponentially
in $j$ by \eqref{eq:sizeg})
we conclude that
$$
\Osc_I w_r \leq C_{10} |b| \Leb(I) \sup_I u_r + C_8 E_I(u_r).
$$
This shows that $(u_r, w_r) \in \CC_b$, as required.
\end{proof}

\begin{proof}[Proof of Lemma~\ref{lemma-L2BV}]
For $m \geq 1$ let $(w_{mn_0}, u_{mn_0})\in\CC_b$ be as in the statement of Proposition~\ref{prop:to-cone}.
Let $v\in \BV$.
Using the  definition of $\|\,\|_b$ norm, 
\begin{align*}
\|\tilde \L_s^{3mn_0} v\|_b&=(1+|b|)^{-1}\Var_Y(\tilde \L_s^{3mn_0} v)+\|\tilde \L_s^{3mn_0}v\|_1\\
&\leq (1+|b|)^{-1}\Var_Y(\tilde \L_s^{3mn_0} v)+\| \tilde \L_s^{2mn_0}(\tilde \L_s^{mn_0} v-w_{mn_0})\|_1+\|\tilde \L_s^{2mn_0}w_n\|_1\\
&\leq (1+|b|)^{-1}\Var_Y(\tilde \L_s^{3mn_0} v)+2C_{10}\rho^{-mn_0}|b|\|v\|_\infty+\beta^{m}\|w_{mn_0}\|_\infty,
\end{align*}
where in the last inequality we have used Proposition~\ref{prop:to-cone} and Lemma~\ref{lemma-L2}. The conclusion follows since
$\|w_{mn_0}\|_\infty\leq \| v\|_\infty$ (as in the statement of Proposition~\ref{prop:to-cone}).~\end{proof}

\subsection{Completing the argument}
\label{sub-Completing the argument}

In this section we complete the proof of Theorem~\ref{th-main} via a couple of lemmas.

\begin{lemma}\label{lemma-complete-1}
There exist  $\eps\in (0,1)$, $A>0$ and $\gamma_1\in (0,1)$ such that 
for all $s=\sigma+ib$, $|\sigma|<\eps$, $|b|\geq\max\{4\pi/D,2\}$ and 
for all $m \geq A\log(1+|b|)$,
\[
\|\tilde \L_s^{3mn_0} v\|_b\leq \gamma_1^{3m}\|v\|_b
\]
for all $v \in \BV(Y)$ satisfying \eqref{eq:H}.
\end{lemma}

\begin{proof}
First, we estimate $(1+|b|)^{-1}\Var_Y(\tilde \L_s^{3mn_0} v)$.
For $m \in \N$, recall from Proposition~\ref{prop:to-cone} 
and Lemma~\ref{lemma-L2} that
\begin{eqnarray*}
\| \tilde \L_s^{2mn_0} v \|_1 &\leq&
\|\tilde \L_s^{mn_0} (\tilde \L_s^{mn_0} v - w_{mn_0}) \|_1 + \| \tilde \L_s^{mn_0} w_{mn_0} \|_1 \\
&\leq& \|\tilde \L_s^{mn_0} (\tilde \L_s^{mn_0} v - w_{mn_0}) \|_\infty
+ \beta^m \| w_{mn_0} \|_\infty \\
&\leq& 2C_{10} \rho^{-mn_0} \| v \|_\infty +  \beta^m \| v \|_\infty
\leq 4\beta^m \| v \|_\infty
\end{eqnarray*}
where we used $C_{10} \rho^{-mn_0} \leq 2\beta^m$.
By Proposition~\ref{prop-LYineq} (which is allowed since $n_0$ is a multiple of $k$)
and recalling that $\Lambda_{\sigma} := \lambda_{2\sigma}^{1/2}/\lambda_{\sigma} \geq 1$,
we compute
\begin{align}\label{eq-varb}
\nonumber \Var_Y(\tilde \L_s^{3mn_0} v)
&\leq\rho^{-mn_0}\Var_Y(\tilde \L_s^{2mn_0} v)+c(1+|b|) \Lambda_{\sigma}^{mn_0}(\|\tilde \L_s^{2mn_0} v\|_1\,\|\tilde \L_s^{2mn_0} v\|_\infty)^{1/2}\\
\nonumber &\leq \rho^{-mn_0}\Var_Y(\tilde \L_s^{2mn_0} v)+2c(1+|b|) \Lambda_{\sigma}^{mn_0}\beta^{m/2}\| v\|_\infty\\
&\leq \rho^{-mn_0}\Var_Y(\tilde \L_s^{2mn_0} v)+2c(1+|b|) \Lambda_{\sigma}^{mn_0}\beta^{m/2}(\Var_Y v+\|v\|_{1}).
\end{align}
where in the last inequality we have used $\|v\|_\infty\leq \Var_Y v+\|v\|_{1}$.
Also by Proposition~\ref{prop-LYineq},
\begin{align*}
\Var_Y(\tilde \L_s^{2mn_0} v) &\leq \rho^{-2mn_0}\Var_Yv+c(1+|b|) \Lambda_{\sigma}^{2mn_0} \|v\|_\infty\\
&\leq  \rho^{-2mn_0}\Var_Yv+c(1+|b|)  \Lambda_{\sigma}^{2mn_0} (\Var_Y v+\|v\|_{1}).
\end{align*}
Plugging the above inequality into~\eqref{eq-varb} we get
\begin{align*}
\Var_Y(&\tilde \L_s^{3mn_0} v)\leq \rho^{-3mn_0}
\Var_Y v+c(1+|b|)(\rho^{-mn_0} \Lambda_{\sigma}^{2mn_0}+2\Lambda_{\sigma}^{mn_0}\beta^{m/2}) (\Var_Y v+\|v\|_{1}).
\end{align*}
Multiplying this $(1+|b|)^{-1}$ and inserting it in
 Lemma~\ref{lemma-L2BV} 
(which relies on the assumption \eqref{eq:H}) gives
\begin{align*}
\|\tilde \L_s^{3mn_0} v\|_b \leq\ &  (1+|b|)^{-1}\rho^{-3mn_0}\Var_Y v+c(\rho^{-mn_0}  \Lambda_{\sigma}^{2mn_0} +2 \Lambda_{\sigma}^{mn_0}\beta^{m/2}) (\Var_Y v+\|v\|_{1})\\
&+(2C_{10}\rho^{-mn_0}|b|+\beta^{m})(\Var_Y v+\|v\|_{1}).
\end{align*}
Hence,
\begin{align*}
\|\tilde \L_s^{3mn_0} v\|_b \leq &\ (1+|b|)^{-1}\Big(\rho^{-3mn_0} +  (1+|b|)(c \Lambda_{\sigma}^{2mn_0} \rho^{-mn_0} \\
& \qquad \qquad \quad  + 
2c \Lambda_{\sigma}^{mn_0}\beta^{m/2}+2C_{10}|b|\rho^{-mn_0} + \beta^m) \Big) \Var_Y v\\
&+(c \Lambda_{\sigma}^{2mn_0} \rho^{-mn_0} + 2c \Lambda_{\sigma}^{mn_0}\beta^{m/2}+2C_{10}|b|\rho^{-mn_0} + \beta^m) \|v\|_1 \\
\leq &\ (1+|b|)^2 (2C_{10} + c) (\Lambda_{\sigma}^{2mn_0} \rho^{-m n_0} + \Lambda_\sigma^{m n_0} \beta^{m/2})\| v \|_b.
\end{align*}
Let $A>0$ be so large that $\gamma_1 := \max\{ \Lambda_{\sigma}^{2n_0}\rho^{-1}, \Lambda_{\sigma}^{n_0} \beta^{1/2}  \} 
\exp( \frac{6\log(2C_0+c)}{A})<1$. Then $(1+|b|)^2 (2C_{10} + c) (\Lambda_{\sigma}^{2mn_0} \rho^{-m n_0} + \Lambda_\sigma^{m n_0} \beta^{m/2})<\gamma_1^m$ for all $m>A\log(1+|b|)$,
and the conclusion follows.
\end{proof}

To complete the proof of Theorem~\ref{th-main} we still need to deal with \BV\ functions violating \eqref{eq:H}.

\begin{lemma}\label{lemma-complete-2}
There exist $\eps\in (0,1)$ and $\gamma_2 \in (0,1)$ such that 
for all $s=\sigma+ib$, $|\sigma|<\eps$, $|b|\geq\max\{4\pi/D,2\}$
and for all $m \geq 1$,
\[
\|\tilde \L_s^{m n_0}v \|_b\leq \gamma_2^m \| v \|_b
\]
for all $v \in \BV(Y)$ violating \eqref{eq:H}.
\end{lemma}

\begin{proof}
By continuity in $\sigma$, $1\leq \Lambda_{\sigma} < \rho^{1/2}$ for all $|\sigma|$ 
sufficiently small.
Then clearly also 
$\gamma_2 := \Lambda_{\sigma}^{n_0} \rho^{-n_0/2} < 1$.
We first treat the case $\sigma \geq 0$, so
by assumption, $\Var_Y v > C_{11} |b|^2 \rho^{m n_0} \|v\|_1$.
Using Proposition~\ref{prop-LYineq} (which is allowed since $n_0$ is a multiple of $k$), we compute that
\begin{align*}
 \Var_Y(\tilde \L_s^{m n_0} v)&\leq\rho^{-m n_0}\Var_Y v+c(1+|b|)
\Lambda_{\sigma}^{mn_0}(\|v\|_1\|v\|_\infty)^{1/2} \\
&\leq \rho^{-m n_0}\Var_Y( v)+c(1+|b|)
\Lambda_{\sigma}^{mn_0}(\|v\|_1(\Var_Y v +\|v\|_1)^{1/2}\\
&\leq \rho^{-m n_0}\Var_Y( v)+c(1+|b|)
\Lambda_{\sigma}^{mn_0} \Big(\frac{\rho^{-m n_0}}{C_{11}|b|^2}\Var_Yv\Big(\Var_Yv+\frac{\rho^{-m n_0}}{C_{11}|b|^2}\Var_Y v\Big)\Big)^{1/2}\\
&\leq  \rho^{-m n_0}\Var_Y v+\frac{c}{C_{11}^{1/2}}\frac{\sqrt{65}}{8} \frac{1+|b|}{|b|} 
\Lambda_{\sigma}^{mn_0} \rho^{-m n_0/2}\Var_Yv \\
&\leq (\rho^{-m n_0} + \frac{1}{8K_2} \frac{3\sqrt{65}}{16} \Lambda_{\sigma}^{mn_0} \rho^{-m n_0/2} ) \Var_Y v,
\end{align*}
where we have used $C_{11}|b|^2 > 64$ and abbreviated $K_2 := \frac{\sup f_{\sigma}}{\inf f_{\sigma}}\frac{\sup f_0}{\inf f_0}$.
Therefore
\[
(1+|b|)^{-1}\Var_Y(\tilde \L_s^{m n_0} v)\leq (1+|b|)^{-1} \frac1{4K_2}\gamma_2^m \Var_Y v
\]
for $m$ sufficiently large. 
By \eqref{eq:LYineq1} at the end of the proof of
Proposition~\ref{prop-LYineq}, 
$$
\|\tilde \L_{\sigma}^{m n_0} |v|\ \|_1\leq \Lambda_{\sigma}^{mn_0}
\frac{\sup f_{\sigma}}{\inf f_{\sigma}}
\Big( \frac{\sup f_{2\sigma}}{\inf f_{2\sigma}} \Big)^{1/2} \,(\|v\|_\infty\|v\|_1)^{1/2}.
$$
Note that $\|\tilde \L_s^{m n_0} v\|_1\leq \|\tilde \L_{\sigma}^{m n_0} |v|\|_1$.
so we have 
\begin{align*}
\|\tilde \L_s^{m n_0} v\|_1
&\leq \Lambda_{\sigma}^{m n_0}\, 
\frac{\sup f_{\sigma}}{\inf f_{\sigma}}
\Big( \frac{\sup f_{2\sigma}}{\inf f_{2\sigma}} \Big)^{1/2}
\Big((\Var_Y v +\|v\|_1)\|v\|_1\Big)^{1/2}\\
&\leq \Lambda_{\sigma}^{m n_0}\,  \frac{\sup f_{\sigma}}{\inf f_{\sigma}}
\Big( \frac{\sup f_{2\sigma}}{\inf f_{2\sigma}} \Big)^{1/2}
\Big( (1+\frac{\rho^{-m n_0}}{C_{11}  |b|^2}) \frac{\rho^{-m n_0}}{C_{11} |b|^2} \Big)^{1/2} \Var_Y v \\
&\leq \frac{\sup f_{\sigma}}{\inf f_{\sigma}}
\Big( \frac{\sup f_{2\sigma}}{\inf f_{2\sigma}} \Big)^{1/2} \frac{\sqrt{65}}{8}C_{11}^{-1/2} |b|^{-1}
\Lambda_{\sigma}^{m n_0} \rho^{-m n_0/2} \ \Var_Y v.
\end{align*}
The choice of $C_{11}$ gives that $\frac{\sup f_{\sigma}}{\inf f_{\sigma}}
\Big( \frac{\sup f_{2\sigma}}{\inf f_{2\sigma}} \Big)^{1/2}<C_{11}^{1/2}/8K_2$. 
Hence, the choice of $\gamma_2$ gives
$\|\tilde \L_s^{m n_0} v\|_1\leq \frac 1{4K_2} (1+|b|)^{-1} \gamma_2^m \Var_Y v$.
Together, $\| \tilde \L_s^{m n_0} \|_b \leq \frac{1}{2K_2} (1+|b|)^{-1} \gamma_2^m \Var_Y v$.

Now if $\sigma < 0$, then the
assumption is $\Var_Y (e^{\sigma \varphi_{m n_0}} v) > 
C_{11} |b|^2 \rho^{m n_0} \| e^{\sigma \varphi_{m n_0}} v\|_1$.
The above computation gives 
$$
\| \tilde \L_s^{m n_0} v \|_b \leq
\frac{\sup f_{\sigma}}{\inf f_{\sigma}}\frac{\sup f_0}{\inf f_0}\
\| \tilde \L_{ib}^{m n_0}( e^{\sigma \varphi_{m n_0}} v) \|_b 
\leq  \frac12 (1+|b|)^{-1}\ \gamma_2^m(2\Var_Y v + \| v \|_1),
$$
where we have used (since $\sigma < 0$)
that $\Var_Y(e^{\sigma \varphi_{m n_0}} v) \leq \Var_Y v + \| v \|_\infty
\leq 2\Var_Y v + \| v \|_1$.
Therefore $\| \tilde \L_s^{m n_0} \|_b \leq (1+|b|)^{-1} \gamma_2^m \| v \|_b$ and this
proves the lemma. 
\end{proof}

\begin{proof}[Proof of Theorem~\ref{th-main}]
Let $\eps \in (0,1)$ be such that the conclusion of 
Lemmas~\ref{lemma-complete-1}, \ref{lemma-complete-2}
and Proposition~\ref{prop-LYineq} hold, and take $\gamma = \max\{ \gamma_1^{1/2},
\gamma_2^{1/2}\}$.
Let $|\sigma| < \eps$, $n \in \N$ and $v \in \BV(Y)$ be arbitrary.
Recall that $|b|\geq \max\{ 4\pi/D,2\}$. 
Let $A$ be the constant used in Lemma~\ref{lemma-complete-1};
without loss of generality, we can assume that $A \log |b| > 3n_0$.
By the proof of Proposition~\ref{prop-LYineq} (see also Remark~\ref{rem-LYtwist}),
there is $A'$ such that the operator norm
\begin{equation}\label{eq:Aprime}
\| \tilde \L_s^{n'} \|_b \leq A' (1+|b|) \quad \text{ for all } |\sigma| < \eps,
b \in \R, n' \in \N.
\end{equation}
Take  
\begin{equation}\label{eq:nminimal}
n \geq 2 \max \left\{ \frac{A}{n_0} \log(1+ |b|) \ ,\  \log(\Lambda_{\sigma}^{-1} \frac{\sup f_{\sigma}}{\inf f_{\sigma}}A'(1+|b|))\right\}.
\end{equation}
Because the contraction in Lemmas~\ref{lemma-complete-1} and 
\ref{lemma-complete-2} happen at different time steps, we carry out the following algorithm:
\begin{enumerate}
\item Let $m_0 \in \N$ be maximal such that $3m_0n_0 \leq n$.
If $m_0 < A\log(1+|b|)$, then continue with Step 4, otherwise continue with Step 2.
\item
If $v$ satisfies ($H_{\sigma, m_0}$), then $\| \tilde \L_s^{3m_0n_0} v \|_b
\leq \gamma^{6m_0} \| v \|_b$ by Lemma~\ref{lemma-complete-1},
and we continue with Step 4.\\
If $v$ does not satisfy ($H_{\sigma, m_0}$), then $\| \tilde \L_s^{m_0n_0 } v \|_b
\leq \gamma^{2m_0} \| v \|_b$ by Lemma~\ref{lemma-complete-2}.
Let $v_1 = \tilde \L_s^{m_0n_0 } v$ and let $m_1 \in \N$ be maximal such that $3m_1n_0  \leq n-m_0n_0 $.\\
If $m_1 < A\log |b|$, then continue with Step 4, otherwise continue with Step 3.
\item
If $v_1$ satisfies ($H_{\sigma, m_1}$), then $\| \tilde \L_s^{3m_1n_0 } v_1 \|_b
\leq \gamma^{6m_0} \| v \|_b$ by Lemma~\ref{lemma-complete-1}.
Therefore 
$$
\| \tilde \L_s^{(3m_1+m_0)n_0 } v \|_b =
\| \tilde \L_s^{3m_1n_0 } v_1 \|_b 
\leq \gamma^{6m_1}\| v_1 \|_b =  \gamma^{3m_1}\| \tilde \L_s^{3m_1n_0} v \|_b
\leq \gamma^{6m_1+2m_0} \| v \|_b,
$$
and we continue with Step 4.\\
If $v_1$ does not satisfies ($H_{\sigma, m_1}$), then $\| \tilde \L_s^{m_1n_0 } v_1 \|_b
\leq \gamma^{2m_0} \| v_1 \|_b$ by Lemma~\ref{lemma-complete-2}.
Let $v_2 = \tilde \L_s^{m_1n_0 } v_1$ and let $m_2 \in \N$ be maximal 
 such that $3m_2n_0  \leq n-(m_0+m_1)n_0$ and repeat Step 3.
Each time we pass through Step 3, we introduce the next integer $m_i$
and $v_i = \tilde\L_s^{m_{i-1}}v_{i-1}$.
As soon as $m_i < A \log(1+|b|)$ we continue with Step 4.
\item Let $p = p(v)$ be the number of times that this algorithm passes through
Step 3. Note that $p < \infty$ because each time Step 3 is taken,
$n-(m_0 + m_1 + \dots + m_i)n_0$ decreases by a factor $2/3$.
Thus we find a sequence
$(m_i)_{i=0}^p$ and we can define
$$
M_p = M_p(v) = \begin{cases}
m_0 + \dots + m_{p-1} + 3m_p,
 & \text{ or }\\
m_0 + \dots + m_{p-1} + m_p, 
\end{cases}
$$
depending on whether $v_{p-1} = \tilde \L_s^{(m_0+ \dots + m_{p-1})n_0} v$ satisfies ($H_{\sigma,m_{p-1}}$) or not. 
In either case we have
$n-M_pn_0 < A \log(1+|b|)$ and
$\| \tilde \L_s^{M_pn_0} v \|_b \leq \gamma^{2M_p} \| v \|_b$.
\end{enumerate}
By \eqref{eq:Aprime}, we have for all $v \in \BV(Y)$
$$
\| \tilde \L_s^nv\|_b = \| \tilde \L_s^{n-M_pn_0} (\tilde \L_s^{M_pn_0} v) \|_b 
\leq \| \tilde \L_s^{n-M_pn_0} \|_b \ \| \tilde \L_s^{M_pn_0} v \|_b 
\leq A' (1+|b|) \gamma^{2M_p} \| v \|_b.
$$
Also $\| \L_s^nv\|_b \leq \lambda_{\sigma}^{-1} \frac{\sup f_{\sigma}}{\inf f_{\sigma}}\| \tilde \L_s^n v\|_b$.
Therefore, using $n-M_pn_0 < A\log |b|$,
\begin{eqnarray*}
\| \L_s^nv\|_b 
& \leq & \lambda_{\sigma}^{-1} \frac{\sup f_{\sigma}}{\inf f_{\sigma}} 
 A' (1+|b|) \gamma^{2M_p} \| v \|_b \\
& \leq & \lambda_{\sigma}^{-1} \frac{\sup f_{\sigma}}{\inf f_{\sigma}} 
 A' (1+|b|) \gamma^{(-A \log |b|)/n_0} \gamma^{2n} \| v \|_b \\
& \leq & \lambda_{\sigma}^{-1} \frac{\sup f_{\sigma}}{\inf f_{\sigma}}
 A' (1+|b|) \gamma^{n/2}  \
\gamma^{(-A \log |b|)/n_0} \gamma^{n/2}\ \gamma^n \| v \|_b
\leq \gamma^n \| v \|_b,
\end{eqnarray*}
since $n$ is chosen large enough as in \eqref{eq:nminimal}.
This completes the proof.
\end{proof}

\appendix

\section{Proof of  Proposition~\ref{prop-LYineq}}
\label{sec-LY}

\begin{proof}[Proof of Proposition~\ref{prop-LYineq}]
Fix $k$ and $\eps$ such that the assumptions of the proposition hold. First, we provide the argument for $n=k$; the conclusion for $n$ a multiple of $k$
will follow by a standard iteration argument.
We note that for each $a\in\alpha^k$ the interval $F^k(a)=[p_a, q_a]$ is the domain of
an inverse branch $h \in \H_k$, which is a contracting diffeomorphism.

Compute that
\begin{align}
\label{eq-var1}
\nonumber \Var_Y \tilde \L_s^k v &\leq \frac{1}{\lambda_{\sigma}^{k}}\frac{1}{\inf f_{\sigma}}\Var\Big( \sum_{h\in\H_k} e^{s\varphi_k\circ h} |h'| (f_{\sigma} v)\circ h\Big)
+ \frac{1}{\lambda_{\sigma}^{k}}\Var\Big(\frac{1}{f_{\sigma}}\Big)\Big\|\sum_{h\in\H_k} e^{s\varphi_k\circ h} |h'| (f_{\sigma} v)\circ h\Big\|_\infty\\
\nonumber &\leq \frac{Q}{\lambda_{\sigma}^{k}}\Var\Big( \sum_{h\in\H_k} e^{s\varphi_k\circ h} |h'| (f_{\sigma} v)\circ h\Big)
+ \Var\Big(\frac{1}{f_{\sigma}}\Big)\Big\|\frac{1}{\lambda_{\sigma}^{k}}\sum_{h\in\H_k} e^{s\varphi_k\circ h} |h'| (f_{\sigma} v)\circ h\Big\|_1\\
&\leq \frac{Q}{\lambda_{\sigma}^{k}}\Var\Big( \L_s^k(f_{\sigma} v)\Big)
+ \Var\Big(\frac{1}{f_{\sigma}}\Big)\sup f_{\sigma}\int \tilde \L_{\sigma}^k |v| d\Leb,
\end{align}
where we abbreviated $Q := \frac{1}{\inf f_{\sigma}}+\Var\Big(\frac{1}{f_{\sigma}}\Big)$.

We estimate the first term in the above equation. Since $v\in \BV(Y)$, $v$ is differentiable Lebesgue-a.e.\ on $Y$ and we let $dv$ denote the generalized derivative;
so, for $[p,q] \subset Y$, we have 
$\Var_Y(1_{[p,q]} v) \leq  \int_p^q |dv| + |v(p)| + |v(q)|$ (see, for instance, ~\cite{Giusti}).

\begin{align}\label{eq:varLnv}
\frac{1}{\lambda_{\sigma}^{k}}\Var\Big( \L_s^k(f_{\sigma} v)\Big) &\leq 
\sum_{h\in\H_k} \Big(\int_{\dom(h)}\frac{\Big| d\Big(e^{s\varphi_k\circ h} |h'| (f_{\sigma} v)\circ h\Big)\Big|}{\lambda_{\sigma}^{k}}
\nonumber \\
& \quad + \frac{|e^{s\varphi_k\circ h}|\ |h'| (f_{\sigma} |v|)\circ h}{\lambda_{\sigma}^k}(p_a)
+ \frac{|e^{s\varphi_k\circ h}| \ |h'| (f_{\sigma} |v|)\circ h}{\lambda_{\sigma}^k }(q_a)\Big) \nonumber \\
& \leq 2\sum_{h\in\H_k}\int_{\dom(h)}\Big|d\Big(\frac{e^{s\varphi_k\circ h} |h'| (f_{\sigma} v)\circ h}{\lambda_{\sigma}^k}\Big)\Big| \nonumber \\
& \quad +
2\sum_{h\in\H_k} \inf_{[p_a,q_a]} \Big|\frac{e^{s\varphi_k\circ h(x)} |h'(x)| (f_{\sigma} v)\circ h(x)}{\lambda_{\sigma}^k}\Big| =: J_1 + J_2.
\end{align}
First, by the finite image property,
$c_0 := \min_{a \in \alpha^k} (q_a-p_a) > 0$ for our fixed $k$. Therefore
\begin{align*}
J_2 &\leq \frac{2}{\min_{a\in\alpha^k}(q_a-p_a)}\sum_{h\in\H_k}\int_{F^k(a)}\frac{e^{\sigma\varphi_k\circ h(x)} |h'(x)| (f_{\sigma} |v|)\circ h(x)}{\lambda_{\sigma}^k} 
\leq \frac{2\sup f_{\sigma}}{c_0} \int_Y \tilde \L_{\sigma}^k|v|\, d\Leb.
\end{align*}
We split the term $J_1$ in \eqref{eq:varLnv} into three terms
\begin{align*} 
\sum_{h\in\H_k}\int_{\dom(h)}\Big|d\Big(\frac{e^{s\varphi_k\circ h} |h'| (f_{\sigma} v)\circ h}{\lambda_{\sigma}^k}\Big)\Big|  \leq I_1+I_2 +I_3
\end{align*}
corresponding to which factor of $\frac{e^{s\varphi_k\circ h} |h'| (f_{\sigma} v)\circ h}{\lambda_{\sigma}^k}$ 
the derivative is taken of.\\
{\bf For $I_1$:} Taking $m=k$ in~\eqref{eq:phin}
\begin{align*}
I_1 := & |\sigma+ib| \sum_{h\in\H_k}\int_{\dom(h)}\frac{e^{\sigma \varphi_k\circ h}(\varphi_k\circ h)'|h'|(f_{\sigma} |v|)\circ h}{\lambda_{\sigma}^k} \, d\Leb \\
\leq  & C'_2 |\eps+b| \sup f_{\sigma}\int_Y \tilde \L_{\sigma}^k|v|\, d\Leb.
\end{align*}
{\bf For $I_2$:} 
Taking $n=k$ in~\eqref{eq:hpp},
\begin{align*}
I_2 = \sum_{h\in\H_k}\int_{\dom(h)} \frac{e^{\sigma\varphi_k\circ h} |h''| (f_{\sigma} |v|)\circ h}{\lambda_{\sigma}^k} \, d\Leb 
\leq C_1\sup f_{\sigma}\int_Y \tilde \L_{\sigma}^k|v|\, d\Leb.
\end{align*}
{\bf For $I_3$:} Due to \eqref{eq:rho2} and using a change of coordinates,
\begin{align*}
I_3 &=  \sum_{h\in\H_k}\int_{\dom(h)}\Big|\frac{e^{\sigma\varphi_k\circ h}|h'|^2d(f_{\sigma} v)\circ h}{\lambda_{\sigma}^k}\Big| \, d\Leb 
\leq \rho^{-3k} \sum_{h\in\H_k}\int_{a} |d(f_{\sigma} v)| \, d\Leb \\ 
& \leq \rho^{-3k}\int_Y |d(f_{\sigma} v)| \, d\Leb
= \rho^{-3k} \Var_Y (f_{\sigma} v)
\leq \rho^{-3k} \sup f_{\sigma} \Var_Y v
+ \rho^{-3k} \Var_Y f_{\sigma} \| v \|_\infty\\
&\leq \rho^{-3k}( \sup f_{\sigma} +\Var_Y f_{\sigma})\Var_Y v
+ \rho^{-3k} \Var_Y f_{\sigma} \int_Y |v| d\Leb,
\end{align*}
where in the last inequality we have used $\|v\|_\infty\leq \Var_Y v +\int |v| d\Leb$.
Putting these together, 
\begin{align*}
\frac{1}{\lambda_{\sigma}^{k}}\Var\Big( \L_s^k(f_{\sigma} v)\Big) & \leq \rho^{-3k}( \sup f_{\sigma} +\Var_Y f_{\sigma})\Var_Y v \\
& \quad  + \rho^{-3k} \Var_Y f_{\sigma} \int_Y |v| d\Leb+ (c_1+C'_2 |b|) \sup f_{\sigma}  \int_Y \tilde \L_{\sigma}^k |v| d\Leb,
\end{align*}
where $c_1= 2c_0^{-1}+C_1+C'_2 \eps$ and $C'_2$ is as in~\eqref{eq:phin}. 
This together with~\eqref{eq-var1}  implies that
\begin{align*}
 \Var_Y \tilde \L_s^k v &\leq \rho^{-3k} Q ( \sup f_{\sigma} +\Var_Y f_{\sigma}) \Var_Yv + \rho^{-3k} \Var_Y f_{\sigma} \int_Y |v| \, d\Leb \\
& \quad + (c_1+\Var\Big(\frac{1}{f_{\sigma}}\Big) + C'_2|b|) \sup f_{\sigma} 
\int_Y \tilde \L_{\sigma}^k |v| d\Leb.
\end{align*}
Given our choice of $\eps$, $c_2:=\Var\Big(\frac{1}{f_{\sigma}}\Big) <\infty$.
By~\eqref{eq:k3}, $c:=\rho^{-2k}Q( \sup f_{\sigma} +\Var_Y f_{\sigma})<1$
and $\rho^{-3k} \Var_Y f_{\sigma}<1$.
Therefore
\begin{align}\label{eq-lsk}
 \Var_Y \tilde \L_s^k v \leq \rho^{-k}\Var_Y v
 + \int_Y |v| d\Leb+ (c_1+c_2+C'_2 |b|)\sup f_{\sigma}  \int_Y \tilde \L_{\sigma}^k |v| d\Leb.
\end{align}
For $n \geq 1$ arbitrary, 
we estimate $\int_Y \tilde \L_{\sigma}^{nk} |v| d\Leb$ applying Cauchy-Schwartz. First, note that
\[
\int_Y \tilde \L_{\sigma}^{nk} |v| d\Leb \leq \Big(\int_Y (\tilde \L_{\sigma}^{nk} |v|)^2 d\Leb\Big)^{1/2}.
\]
Recall that $\Lambda_{\sigma} = \frac{\lambda_{2\sigma}^{1/2}}{\lambda_{\sigma}}$.
Then
\begin{align*}
\int & (\tilde \L_{\sigma}^{nk} |v|)^2 d\Leb =\int (\lambda_{\sigma}^{nk} f_{\sigma})^{-2}\Big(\sum_{h\in\H_{nk}} e^{\sigma\varphi_{nk}\circ h} |h'| (f_{\sigma} |v|)\circ h\Big)^2\, d\Leb\\
&=\int (\lambda_{\sigma}^{nk} f_{\sigma})^{-2}\Big(\sum_{h\in\H_{nk}} (e^{\sigma\varphi_{nk}\circ h} |h'|^{1/2} (f_{\sigma} |v|)^{1/2}\circ h)(|h'|^{1/2} (f_{\sigma} |v|)^{1/2}\circ h)\Big)^2\, d\Leb\\
&\leq \lambda_{\sigma}^{-2nk}(\inf f_{\sigma}^2)^{-1}\int \Big(\sum_{h\in\H_{nk}} e^{2\sigma\varphi_{nk} \circ h} |h'|(f_{\sigma} |v|)\circ h\Big)
\Big(\sum_{h\in\H_{nk}} |h'|(f_{\sigma} |v|)\circ h\Big)\, d\Leb\\
&\leq \Lambda_{\sigma}^{2nk}
\Big( \frac{\sup f_{\sigma}}{\inf f_{\sigma}} \Big)^2 
\frac{\sup f_{2\sigma}}{\inf f_{2\sigma}} \| v \|_\infty 
\int \Big( \sum_{h \in \H_{nk}} \frac{e^{2\sigma \varphi_{nk} \circ h}}{\Lambda_{2\sigma}^{nk}f_{2\sigma}}
|h'| f_{2\sigma} \circ h \Big) \ \Big( \sum_{h \in \H_{nk}} |h'| |v| \circ h \Big) d\Leb 
\\ 
&\leq \Lambda_{\sigma}^{2nk} \Big( \frac{\sup f_{\sigma}}{\inf f_{\sigma}} \Big)^2 
\frac{\sup f_{2\sigma}}{\inf f_{2\sigma}} \| v \|_\infty \| v \|_1.
\end{align*}
Thus,
\begin{equation}\label{eq:LYineq1}
\int_Y \tilde \L_{\sigma}^{nk} |v| d\Leb \leq 
\Lambda_{\sigma}^{nk} \frac{\sup f_{\sigma}}{\inf f_{\sigma}}
\Big( \frac{\sup f_{2\sigma}}{\inf f_{2\sigma}} \Big)^{1/2}
\,(\|v\|_\infty\|v\|_1)^{1/2}.
\end{equation}
The above together with~\eqref{eq-lsk} implies that
\begin{align}\label{eq-var4}
 \nonumber\Var_Y \tilde \L_s^{nk} v &\leq \rho^{-k}\Var_Y  \tilde \L_s^{(n-1)k} v
 + (1+c_1+c_2+C'_2 |b|) \Lambda_{\sigma}^{nk}
 \frac{\sup f_{\sigma}}{\inf f_{\sigma}}
\Big( \frac{\sup f_{2\sigma}}{\inf f_{2\sigma}} \Big)^{1/2}
\,(\|v\|_\infty\|v\|_1)^{1/2} \\
&\leq \rho^{-k}\Var_Y \tilde \L_{\sigma}^{(n-1)k} v+ c_3(1+|b|) \Lambda_{\sigma}^{nk} (\|v\|_\infty\|v\|_1)^{1/2},
\end{align}
for $c_3:=\max\{1+c_1+c_2,\,C'_2\} \frac{\sup f_{\sigma}}{\inf f_{\sigma}}
\Big( \frac{\sup f_{2\sigma}}{\inf f_{2\sigma}} \Big)^{1/2}$.
Iterating~\eqref{eq-var4}, we obtain that
\begin{align*}
 \Var_Y \tilde \L_s^{nk} v &\leq \rho^{-nk}\Var_Y v
 +c(1+|b|) \Lambda_{\sigma}^{nk} (\|v\|_\infty\|v\|_1)^{1/2},
\end{align*}
for any $n\geq 1$,
where $c:=c_3 \sup f_{\sigma} \sum_{j=0}^{n-1} (\rho \Lambda_{\sigma})^{-jk}$. 
This ends the proof.~\end{proof}

\begin{remark}\label{rem-LYtwist} A similar, but much more simplified, 
argument to the one used in the proof of Proposition~\ref{prop-LYineq} shows that the non-normalized twisted transfer operator satisfies 
$\Var_Y (\L_s^n v) \leq c_1\rho_0^{-n} \Var_Y v + c_2(1+|b|) \|v\|_\infty$, for all $n\geq 1$,
some $\rho_0>1$, $c_1, c_2>0$, for all $b\in\R$ and all $|\sigma|<\eps$, for any $\eps\in (0,1)$.
\end{remark}

\begin{remark}\label{rem:imaginary_axis}
If $\sigma = 0$, so when working on the imaginary axis, we can
get the standard Lasota-Yorke inequality
$\Var_Y \L_{ib}^n v \leq \rho^{-n} \Var_Y v + c_4 (1+|b|) \| v \|_1$.
\end{remark}

\begin{proof}[Proof of Proposition~\ref{prop-continuity}]
Take $\eps = \eps_0^2$. Without loss of generality, set $0 \leq |\sigma_2| \leq |\sigma_1|<\eps$ and take $b\in \R$,
\begin{align*}
\| \L_{s(\sigma_1+ib_1} v - \L_{(\sigma_2+ib_2} v\|_1 & =
\int_Y \left|\sum_{h \in \H} \Big(e^{ (\sigma_1+ib_1) \varphi \circ h} - e^{(\sigma_2+ib_2)\varphi \circ h} \Big) |h'| v \circ h \right| \, d\Leb \\
&\leq \| v \|_\infty  \int_Y \sum_{h \in \H} e^{\sigma_1 \varphi \circ h} |h'| \Big(1 - e^{(\sigma_2-\sigma_2) \varphi \circ h} \Big) \, d\Leb.
\end{align*}
Because the function $x \mapsto e^{-(\eps_0-\sigma_1) x} x$ 
assumes its maximum value $e^{-1}(\eps_0-\sigma)^{-1}$
at $x = (\eps_0-\sigma)^{-1}$, we have
$$
e^{\sigma_1 \varphi \circ h} \Big(1 - e^{(\sigma_2-\sigma_2) \varphi \circ h}\Big) 
\leq e^{\eps_0 \varphi \circ h} |\sigma_1 - \sigma_2|
e^{-( \eps_0-\sigma_1)\varphi \circ h} \varphi \circ h\leq 
\frac{ e^{\eps_0 \varphi \circ h}}{e(\eps_0-\sigma)}.
$$
Plugging this into the above, we find
$$
\int_Y \sum_{h \in \H} \Big(e^{ \sigma_1 \varphi \circ h} - e^{\sigma_2 \varphi \circ h} \Big) |h'| v \circ h \, d\Leb
\leq\frac{ \| v \|_\infty}{e(\eps_0-\sigma)}
\int_Y \sum_{h \in \H} e^{\eps_0 \varphi \circ h} |h'| \, d\Leb
\leq \frac{C_3  \| v \|_\infty}{e(\eps_0-\sigma)}.
$$
To estimate $\Var_Y( \L_{s(\sigma_1+ib_1} v - \L_{(\sigma_2+ib_2} v)$,
we work as in the Proof of Proposition~\ref{prop-LYineq},
and use the above estimate on the $L^1$-norm.
As such we obtain
$$
\Var_Y( \L_{s(\sigma_1+ib_1} v - \L_{(\sigma_2+ib_2} v)
\leq |\sigma_1-\sigma_2| \eps_0^{-1} (C' \Var_Y v + C'' \| v \|_\infty)
\leq C|\sigma_1-\sigma_2| \eps_0^{-1}\| v \|_{\BV}
$$
for some $C > 0$ as required.
\end{proof}

\section{Proofs of Lemmas~\ref{lem:jump_fsigma} and \ref{lem:G}}
\label{sec:Hofbauer}

\begin{proof}[Proof of Lemma~\ref{lem:jump_fsigma}]
Recall that $f_{\sigma}$ is an eigenfunction for the non-normalized twisted
transfer operator 
$\L_{\sigma}$, so $\frac{1}{\lambda_{\sigma}} \L_{\sigma}^r f_{\sigma}(x) = f_{\sigma}(x)$
for every $r \in \N$ and $x \in Y$.
Therefore, for $r \in \N$ arbitrary, we have
\begin{eqnarray}\label{eq:ev}
\frac{1}{\lambda_{\sigma}^r} \L_{\sigma}^r 1(x) & = &
\frac{1}{\lambda_{\sigma}^r}
\sum_{h \in \H_r, x \in \dom(h)} |h'(x)|\ e^{\sigma \varphi_r \circ h(x)} 
\nonumber \\
&\leq& \sum_{h \in \H_r,  x \in \dom(h)} \frac{|h'(x)|\ e^{\sigma \varphi_r \circ h(x)} f_{\sigma} \circ h(x)}{\lambda_{\sigma}^r f_{\sigma}(x)}  \frac{\sup f_{\sigma}}{\inf f_{\sigma}}
 \leq \frac{\sup f_{\sigma}}{\inf f_{\sigma}}
\end{eqnarray}
for all $x \in Y$, and similarly $\frac{1}{\lambda_{\sigma}^r} \L_{\sigma}^r 1(x) 
\geq \frac{\inf f_{\sigma}}{\sup f_{\sigma}}$.
Hence the Cesaro means converge to the fixed point
with unit $L^1$-norm:
$$
\lim_{n \to \infty} \frac1n \sum_{r=0}^{n-1} \L_{\sigma}^r 1 = \frac{f_{\sigma}}{\int_Y f_{\sigma} \ d\Leb}.
$$
If $x \notin X_\infty$, then $\L_{\sigma}^r 1$ is continuous at $x$ for 
all $r \in \N$, and so is $f_{\sigma}$.
Now for $x \in X'_j$ take $r \geq j$.
The discontinuity of $\L_{\sigma}^r1$ at $x \in X'_j$ is created
by non-onto branches of $F^r$, and there exist $y \in X'_1$
and an inverse branch $\tilde h \in \H_{j-1}$ such that
 $y = \tilde h(x)$. The jump-size of $\tilde \L_{\sigma}^r 1$ at $x$
can be expressed as a sum of $h \in \H_{r-(j-1)}$ which in the summand
is composed with $\tilde h$.
Then, using \eqref{eq:rho2} and also \eqref{eq:ev} for iterate $r-(j-1)$
to estimate the sum in brackets below:
\begin{eqnarray*} 
\Size \frac{1}{\lambda_{\sigma}^r} \L_{\sigma}^r 1(x) 
&\leq& \frac{1}{\lambda_{\sigma}^r}
\sum_{h \in \H_{r-(j-1)}, y \in \dom(h)} \!\!\!\!\!\! |(h \circ \tilde h)'(x)|\
e^{\sigma \varphi_{r-(j-1)} \circ h \circ \tilde h(x) + \sigma \varphi_{j-1} \circ \tilde h(x)} \\
&=& \Big( \sum_{h \in \H_{r-(j-1)}, y \in \dom(h)} \!\!\!\!\!\! 
\frac{|h'(y)|\ e^{\sigma \varphi_{r-(j-1)} \circ h(y)}}{\lambda_{\sigma}^{r-(j-1)}} \Big) \ 
\frac{|\tilde h'(x)| \ e^{\sigma \varphi_{j-1} \circ \tilde h(x)}}{\lambda_{\sigma}^{j-1}}  \\
&\leq & \frac{\sup f_{\sigma}}{\inf f_{\sigma}} \rho^{-3(j-1)}.
\end{eqnarray*}
By taking the Cesaro limit we obtain
statement 1.\ of the lemma for $C_7 = \rho^3 \frac{\sup f_{\sigma}}{\inf f_{\sigma}}$.

Now for statement 2.\  let $I \subset Y$ be an arbitrary interval, and let $J$ denote a component of $I \setminus X_r$. Note that if $h \in \H_r$ is such
that $J \cap \dom(h) \neq \emptyset$, then $\dom(h) \supset J$.
The oscillation $\Osc_I(\frac{1}{\lambda_{\sigma}^r} \L_{\sigma}^r1)$ is bounded 
by the sum of jump-sizes of discontinuities 
in $I$ added to the sum of the oscillations  
$\Osc_J(\frac{1}{\lambda_{\sigma}^r} \L_{\sigma}^r1)$ on the components $J$ 
of $I \setminus X_r$.
For the latter, we have using formulas \eqref{eq:phin}, \eqref{eq:hpp}
and \eqref{eq:ev}:
\begin{eqnarray*}
\Osc_J (\frac{1}{\lambda_{\sigma}^r} \L_{\sigma}^r 1)
&\leq & \frac{1}{\lambda_{\sigma}^r}\sum_{h \in \H_r} \int_{J \cap \dom(h)}
|(e^{\sigma \varphi_r \circ h(\xi)} |h'(\xi)|)'| \ d\xi\\
&\leq & \frac{1}{\lambda_{\sigma}^r}\sum_{h \in \H_r} \int_{J \cap \dom(h)}
\Big(|\sigma|\ |(\varphi_r \circ h)'(\xi)| e^{\sigma \varphi_r \circ h(\xi)} 
+  e^{\sigma \varphi_r \circ h(\xi)} |h''(\xi)|\Big)\ d\xi \\
&\leq &  \int_J \sum_{h \in \H_r, J \cap \dom(h) \neq \emptyset} 
\frac{ e^{\sigma \varphi_r \circ h(\xi)} |h'(\xi)|}{\lambda^r_\sigma} (|\sigma| C'_2 + C_1)\ d\xi \\
&\leq &  (\eps C'_2 + C_1) 
\int_J \frac{\sup f_{\sigma}}{\inf f_{\sigma}} \ d\xi 
= (\eps C'_2 + C_1) \frac{\sup f_{\sigma}}{\inf f_{\sigma}} \  \Leb(J).
\end{eqnarray*}
Recall from Remark~\ref{rem:lambda}
that $\frac{\sup f_{\sigma}}{\inf f_{\sigma}} \leq C_5$. 
Summing over all components $J$ of $I\setminus X_r$  gives 
$$
\Osc_I(\frac{1}{\lambda_{\sigma}^r} \L_{\sigma}^r1)
\leq 
(\eps C'_2 + C_1) C_5\  \Leb(I) +  \rho^3 C_5 \sum_{j \leq r} \sum_{x \in X'_j \cap I} \rho^{-3j}.
$$
For the Cesaro limit, we get
$\Osc_I(f_{\sigma}) \leq C_6 \mu(I) + C_7 E_I(f_{\sigma})$
for $C_6 = (\eps C'_2 + C_1) C_5$ and $C_7 = \rho^3C_5$ as required.
This implies also the formula for $\Osc(1/f_{\sigma})$,
adjusting the constants $C_6$ and $C_7$ if necessary.
\end{proof}

Before stating the next lemma, we recall that
$K = \min\{ \Leb(F(a)) : a \in \alpha\}$
and that  $\delta_0 = \frac{K(\rho_0-2)}{5e^{C_1}\rho_0}$.
Since $F$ is topologically mixing, there is $k_1 \in \N$ such 
that $F^{k_1}(I) \supset Y$ for all intervals $I$ of length $\Leb(I) \geq \delta_0$.

\begin{lemma}\label{lem:eta1}
There is $\eta_1 \in (0,1)$ such that for every $z \in Y$
and $\tau > 0$ the following property holds:
For every 
$n \geq k_1 + \frac{\log(2K(\rho_0-2)/(e^{C_1} \rho_0 \tau))}{\log(\rho_0/2)}$
and every interval $J$ of length $\Leb(J) > \tau$,
$$
\Leb(\bigcup_{\tilde a \in J_z} \tilde a) \geq \eta_1 \Leb(J)
\quad \text{ for }\
J_z = \{ \tilde a \in \alpha^n : \tilde a \subset J \text{ and } z \in F^n(\tilde a) \}.
$$
\end{lemma}

\begin{proof}
By the choice of $k_1$, there is a finite
collection $\Omega$ of $k_1$-cylinders such that
for each $z \in Y$ and each $I$ with $\Leb(I) \geq \delta_0$, there is 
$\omega \in \Omega$, $\omega \subset I$, such that $z \in F^{k_1}(\omega)$.
Let $\gamma_0 := \min\{ \frac{\Leb(\omega)}{2\delta_0} : \omega \in \Omega\} > 0$.

For $y \in Y$, define $r_j(y) = d(F^j(y), \partial F^j(a))$, where $a \in \alpha^j$ is the $j$-cylinder containing $y$.
Take $J$ an arbitrary interval of length $\Leb(J) \geq \tau$, and define
$Z^j_\delta = \{ y \in J : r_j(y) \leq \delta\}$. We derive $\Leb(Z^{j+1}_\delta)$ from
$\Leb(Z^j_\delta)$ as follows.
If $a \in \alpha^j$, $W = F^j(a)$ and $a' \in \alpha$ are such that
$\partial W \cap a' \neq \emptyset$, then the points
$\{ z \in F(W \cap a') : d(z, \partial F(W \cap a')) \leq \delta\}$
pull back to at most two intervals in $W \cap a'$ of combined length
$\leq 2\delta/\rho_0$, and this contributes $2\Leb(Z^j_{\delta/\rho_0})$
to $\Leb(Z^{j+1}_\delta)$.
For the cylinders $a' \in \alpha$ that are contained in $W$,
we recall that $\Leb(F(a')) \geq K$. By the distortion bound from
\eqref{eq:adler} we find
$\Leb(Z^{j+1}_\delta \cap F^{-j}(a')) \leq \frac{2e^{C_1}\delta}{K} \Leb(a)$.
Combining this (and summing over all such $a$), we get the recursive relation
$\Leb(Z^{j+1}_\delta) \leq 2\Leb(Z^j_{\delta/\rho_0}) +  \frac{2e^{C_1}\delta}{K} \Leb(J)$.
This gives
$$
\Leb(Z^j_\delta) \leq 2^j \Leb(Z^0_{\delta/\rho_0^j}) + \frac{2e^{C_1}\delta}{K} 
\sum_{i=0}^{j-1} \Big(\frac{2}{\rho_0}\Big)^i
\leq \Big( \Big(\frac{2}{\rho_0}\Big)^j \frac{\delta}{\Leb(J)} + \frac{2e^{C_1} \rho_0}{K(\rho_0-2)} \delta \Big) \Leb(J).
$$
Take $\delta = \delta_0$ and $j \geq \frac{\log(10 \delta_0/\Leb(J))}{\log(\rho_0/2)} = \frac{\log(2K(\rho_0-2)/(e^{C_1} \rho_0 \tau))}{\log(\rho_0/2)}$ (so that $(\frac{2}{\rho_0})^j \frac{\delta_0}{\Leb(J)} \leq \frac{1}{10}$). Then
\begin{equation}\label{eq:growth}
\Leb\left(y \in J : r_j(y) \geq \delta_0 \right) =\Leb(J)-\Leb(Z^j_{\delta_0})
\geq  \Leb(J)-\frac12 \Leb(J)=\frac12 \Leb(J).
\end{equation}
Let $B_{j,J}$ be the collection of $a \in \alpha^j$, $a \subset J$ such 
that there is $y \in a$ with $r_j(y) \geq \delta_0$.
This means by \eqref{eq:growth} that $\Leb(\cup_{a \in B_{j,J}} a) \geq \frac12 \Leb(J)$
and $1 \geq \Leb(F^j(a)) \geq 2\delta_0$ for each $a \in B_{j,J}$.
Take $z \in Y$ and $n = j+k_1$.
It follows that there is an $n$-cylinder $\tilde a \subset a$ such that
$F^j(\tilde a) = \omega \in \Omega$ and $z \in F^{k_1}(\omega)$.
By boundedness of distortion
$$
\frac{\Leb(\tilde a)}{\Leb(a)} \geq e^{-C_1} \frac{\Leb(F^j(\tilde a)}{\Leb(F^j(a))}
\geq e^{-C_1} \frac{\Leb(\omega)}{2\delta_0} \geq \gamma_0 e^{-C_1}.
$$
Hence $\Leb(\cup_{a \in J_z} a) \geq \gamma_0 e^{-C_1} \Leb(\cup_{a \in B_{n,J}} a) \geq
\frac{\gamma_0}{2 e^{C_1}} \Leb(J)$,
proving the lemma for $\eta_1 := \frac{\gamma_0}{2e^{C_1}}$.
\end{proof}

Now we are ready for the proof of Lemma~\ref{lem:G}, which uses assumption \eqref{eq:k}.

\begin{proof}[Proof of Lemma~\ref{lem:G}]
We will apply Lemma~\ref{lem:eta1} for $J = p$, an arbitrary element 
of $\P_k$.
Set for $C_9 = \eta_1 e^{-C_1}/2$.
Assumption \eqref{eq:k} gives $\Leb(p) \geq 12\rho^{-k}$.
Since $n = 2k$, we have $j := n-k_1 \geq k$.
Therefore
$(\frac{2}{\rho_0})^j \frac{\delta_0}{\Leb(p)} \leq 
\frac{2^k \rho^{-3k} \delta_0}{12} < \frac{1}{12}$, and hence
\eqref{eq:growth} implies that 
$\Leb(y \in p: r_j(y) \geq \delta_0) \geq \frac12 \Leb(p)$.

Recall that $B_{j,p} \supset \{ a \in \alpha^j : a \subset p, r_j(y) \geq \delta_0
\text{ for some } y \in a\}$, so $F^j(a) \geq 2\delta_0$
for each $a \in B_{j,p}$. In particular, such $a$ contains an 
$\tilde a \in \alpha^n$ such that $z\in F^n(\tilde a)$,
and $\Leb(\cup_{a \in B_{j,p}} \tilde a) \geq \eta_1 \Leb(p)$
with $\eta_1$ as in Lemma~\ref{lem:eta1}.
Let $B^*_{j,p}$ be a finite subcollection of $B_{j,p}$ such that
$\Leb(\cup_{a \in B^*_{j,p}} \tilde a) \geq \frac23 \eta_1 \Leb(p)$,
and let $h_{\tilde a} : F^n(\tilde a) \to \tilde a$ denote the
corresponding inverse branches.

Using the continuity of $\sigma \mapsto \lambda_{\sigma}$ and 
$\sigma \mapsto e^{\sigma \varphi_n \circ h_{\tilde a}(z)}$ for all
$a \in B^*_{j,p}$, $j \leq 4k-k_1$ and $p \in \P_k$,
we can choose $\eps$ so small that
$\frac{1}{\lambda_{\sigma}^n} |h'_{\tilde a}(z)| e^{\sigma \varphi_n \circ h_{\tilde a}(z)} \geq \frac34 |h'_{\tilde a}(z)|$
for all $a \in B^*_{j,p}$ and all $|\sigma| < \eps$.
Therefore
\begin{eqnarray*}
\frac{1}{\lambda_{\sigma}^n} \sum_{\stackrel {h \in \H_n, z \in \dom(h)}{\range(h) \subset p}}
|h'(z)| e^{\sigma \varphi_n \circ h(z)} 
&\geq & \frac{1}{\lambda_{\sigma}^n} \sum_{a \in B^*_{j,p}}
|h'_{\tilde a}(z)| e^{\sigma \varphi_n \circ h_{\tilde a}(z)} \geq \frac34 \sum_{a \in B^*_{j,p}}
|h'_{\tilde a}(z)| \\
&\geq & \frac34 \sum_{a \in B^*_{j,p}} e^{-C_1} \frac{\Leb(\tilde a)}{\Leb(F^n(\tilde a))}
\geq \frac{\eta_1  \Leb(p)}{2e^{C_1}}.
\end{eqnarray*}
This finishes the proof. 
\end{proof}

\section{A technical result for the proof of Proposition~\ref{prop:to-cone}}
\label{sec-Varvar}

In this subsection we will use the
generalised \BV\ seminorm $\var_Y v$ introduced by 
Keller \cite{Keller} because it compares 
more easily with $\| \ \|_1$ than $\Var_Y$ does.
To be precise, we define
$$
\var_Y v = \sup_{0 < \kappa < 1} \frac{1}{\kappa}
\int_Y \Osc(v, B_\kappa(x)) \, d\Leb,
$$
where $\Osc(v, B_\kappa(x)) = \sup_{y, y' \in B_\kappa(x)} | v(y) - v(y')|$
(also for complex-valued functions).

\begin{lemma}
In dimension one, $\Var_Y$ and $\var_Y$ are equivalent seminorms.
More precisely, for all $v \in \BV(Y)$ we have
\begin{equation}\label{eq:equiv}
\frac12 \Var_Y v \leq \var_Y v \leq 3 \Var_Y v.
\end{equation}
\end{lemma}

\begin{proof}
\cite[Lemma 1]{BK} states that $\Var_Y v \leq 2 \var_Y v$.
For the other inequality, choose $\kappa \in (0,1)$ and partition $Y$ into 
half-open intervals $J$ of length $|J| \leq \kappa$.
For each such $J$, let $J'$ and $J''$ denote its left and right neighbour.
Then
\begin{eqnarray*}
\frac{1}{\kappa} \int_Y \Osc(v, B_\kappa(x)) \, d\Leb
&=& \frac{1}{\kappa} \sum_J \int_J \Osc(v, B_\kappa(x)) \, d\Leb
\leq \frac{1}{\kappa} \sum_J \Leb(J) \Osc_{J \cup J' \cup J''} v \\
&\leq& \sum_J \Osc_{J \cup J' \cup J''} v \leq 3 \Var_Y v.
\end{eqnarray*}
Both inequalities together prove \eqref{eq:equiv}.
\end{proof}

Recall that $K := \min\{ |F(a)| : a \in \alpha\}$.

\begin{lemma}\label{lem:Varvar} 
Let $v \in \BV(Y)$ such that $\Var_Y v \leq K_0 \| v \|_1$ for some $K_0 > 1$.
Choose $\eta_1 \in (0,1)$ such that Lemma~\ref{lem:eta1} holds
and take $K_1 =  6e^{C_1}/\eta_1$. 
Let
$$
r_0 := \max\Big\{ k, k_1 + \Big(\log \frac{108 K_0 K (\rho_0-2)}{e^C_1\rho_0} \Big)/ \log \frac{\rho_0}{2} \Big\}.
$$
Then for every $r > r_0$ and all $I_r \in Q_r$,
$$
\|v\|_1 \leq \frac{K_1}{\Leb(I_r)} \int_{F^{-r}(I_r)} |v|\, d\Leb.
$$
\end{lemma}

\begin{proof}[Proof of Lemma~\ref{lem:Varvar}]
Fix $\kappa_1:= (18K_0)^{-1}$. Since we assumed that $K_1 > 6e^{C_1}/\eta_1$
we have $(1-\frac{4e^{C_1}}{\eta_1 K_1}) \geq 6K_0 \kappa_1$.
Let $E$ be a partition of $Y$ into half-open intervals $J = [p,q)$
of length $\frac{\kappa_1}{3} \leq \Leb(J) \leq \frac{\kappa_1}{2}$.
Next recall that  $K := \min\{ |F(a)| : a \in \alpha\}$ and take 
$r > r_0$.
Note that this $r_0$ is the bound from 
Lemma~\ref{lem:eta1} with $\tau = \kappa_1/3 = 1/(54K_0)$.

We prove the lemma by contradiction, so assume
that there exists $I_r \in Q_r$  such that 
$\| v \|_1 > \frac{K_1}{\Leb(I_r)} \int_{F^{-r}(I_r)} |v|\, d\Leb$.
Define
$$
M(I_r) = \Big\{ J \in E \ : \  \int_{F^{-r}(I_r) \cap J}|v| \, d\Leb \leq \frac{2 \Leb(I_r)}{K_1} \int_J |v| \, d\Leb \Big\}.
$$
If $\sum_{J \in M(I_r)} \int_J |v|\, d\Leb < \frac12 \|v\|_1$
(so $\sum_{J \notin M(I_r)} \int_J |v|\, d\Leb > \frac12 \|v\|_1$),
then we have
\begin{eqnarray*}
\int_{F^{-r}(I_r)} |v| \, d\Leb &\geq& \sum_{J \notin M(I_r)} \int_{F^{-r}(I_r) \cap J} |v| \, d\Leb \\
&>& \frac{2\Leb(I_r)}{K_1} \sum_{J \notin M(I_r)} \int_J |v| \, d\Leb 
> \frac{2\Leb(I_r)}{K_1} \frac12 \int_Y |v| \, d\Leb, 
\end{eqnarray*}
contradicting our choice of $I_r$. Therefore, it remains to deal with the case
\begin{equation}\label{eq:M}
\sum_{J \in M(I_r)} \int_J |v|\, d\Leb > \frac12 \|v\|_1.
\end{equation}
Recall that $e^{C_1}$ is a uniform distortion
bound for the inverse branches of $F^r$.
Let $z$ be the middle point of $I_r$ and $J_z = \{ a \in \alpha^r : a \subset J, z \in F^r(a)\}$.
This means in particular that $\frac{\Leb(F^r(a) \cap I_r}{\Leb(F^r(a))} \geq \frac12 \Leb(I_r)$ for each $a \in J_z$.
By Lemma~\ref{lem:eta1}, $\Leb(\cup_{a \in J_z} a) \geq \eta_1 \Leb(J)$. 
This gives 
\begin{eqnarray*}
\int_{F^{-r}(I_r) \cap J} |v| \, d\Leb
&\geq& \inf_J |v| \Leb(F^{-r}(I_r) \cap J) 
\geq \inf_J |v| \sum_{a \in J_z} \Leb(F^{-r}(I_r) \cap a)  \\
&\geq&
\inf_J |v| \sum_{a \in J_z} e^{-C_1} \frac{\Leb(F^r(a) \cap I_r)}{\Leb(F^r(a))}  \Leb(a) \\
&\geq&
\frac{\inf_J |v|}{2e^{C_1}}\sum_{a \in J_z} \Leb(a) \Leb(I_r)
\geq \frac{\eta_1 \inf_J |v|}{2e^{C_1}} \Leb(J) \Leb(I_r).
\end{eqnarray*}
Hence for each $J \in M(I_r)$,
\begin{eqnarray*}
\Leb(J) \Leb(I_r)\inf_J |v| &\leq& \frac{2e^{C_1}}{\gamma} \int_{F^{-r}(I_r) \cap J} |v|\, d\Leb \\
&\leq& \frac{4e^{C_1}}{\eta_1 K_1} \Leb(I_r) \int_J |v| \, d\Leb
\leq  \frac{4e^{C_1}}{\eta_1 K_1}  \Leb(J) \Leb(I_r) \sup_J |v|
\end{eqnarray*}
and therefore
$\inf_J |v| \leq \frac{4e^{C_1}}{\eta_1 K_1} \sup_J |v|$ and
\begin{equation}\label{eq:OscJ}
\Osc_J v \geq \Osc_J |v| \geq (1-\frac{4e^{C_1}}{\eta_1 K_1}) \sup_J |v|.
\end{equation} 
Recall that by the choice of $\kappa_1$, $ \kappa_1^{-1}(1-\frac{2e^{C_1}}{\eta_1 K_1}) \geq 6K_0$. Bounding the $\sup$ from below using \eqref{eq:OscJ}, 
we obtain
$$
\sup_{0< \kappa < 1} \frac{1}{\kappa} \int_J \Osc(v, B_\eps(x)) \, d\Leb
\geq \Leb(J) \kappa_1^{-1} (1-\frac{4e^{C_1}}{\eta_1 K_1}) \sup_J |v|
\geq 6K_0 \Leb(J) \sup_J |v|.
$$
By the second inequality in \eqref{eq:equiv},
\begin{eqnarray*}
\Var_Y v &\geq& \frac13 \var_Y v \geq
\frac{1}{3\kappa} \sum_{J \in E} \int_J \Osc(v, B_{\kappa}(x)) \, d\Leb\\
&\geq& \frac13 \sum_{J \in M(I_r)} 6K_0 \Leb(J)\sup_J |v| 
\geq 2K_0  \sum_{J \in M(I_r)} \int_J |v| \, d\Leb.
\end{eqnarray*}
Finally \eqref{eq:M} gives
$\Var_Y v > K_0 \int_Y |v| \, d\Leb = K_0 \| v \|_1$.
This contradicts the assumption of the lemma, completing the proof.
\end{proof}

\section{Proof of Theorem~\ref{thm-decay} }\label{sec:expo}

The proof of Theorem~\ref{thm-decay} follows closely the argument used in~\cite[Proof of Theorem 2.1]{AM} 
with obvious required modifications. As in~\cite{AM}, the conclusion follows once we show
that the Laplace transform $\hat\rho(s):=\hat\rho(s)(v,w):=\int_0^\infty e^{st} \rho_{t}(v,w)\, dt$ behaves as described
in the result below.

\begin{lemma}\label{lemma-AM217}
There exists $\eps>0$ such that $\hat\rho(s)$ is analytic on $\{\Re s>\eps\}$
for all $v\in F_{\BV,2}(Y^\varphi)$ and $w\in L^\infty(Y^\varphi)$. Moreover,
there exists $C>0$ such that $|\hat\rho(s)|\leq C (1+|b|^{1/2})\|v\|_{\BV, 2}\|w\|_{\infty}$,
for all $s=\sigma+ib$ with $\sigma\in [0, \frac 12\eps]$.
\end{lemma} 

The proof of Theorem~\ref{thm-decay} given Lemma~\ref{lemma-AM217} is standard, relying on the
formula $\rho_{t}(v,w)=\int_\Gamma e^{-st}\hat\rho(s)\, ds$, where $\Gamma=\{\Re s=\eps/2\}$; it
goes, for instance, exactly the same as ~\cite[Proof of Theorem 2.1]{AM} given ~\cite[Lemma 2.17]{AM}, so we omit this.

The proof of Lemma~\ref{lemma-AM217} uses three ranges of $n$ and $b$: 
i) $n\leq A\log |b|$, $|b|\geq 2$ with $A$ as in Theorem~\ref{th-main}, ii) $|b|\geq \max\{4\pi/D, 2\}$ and iii) 
$0 < |b| < \max\{4\pi/D, 2\}$.
The first two regions go almost word by word as in \cite[Lemma 2.17]{AM}.
For the third region, the part of the proof in \cite{AM}
where the standard form of Lasota-Yorke inequality of $\tilde \L_s$ is used
doesn't apply (in our case $\|\tilde \L_{\sigma+ib}\|_1$ with $\sigma>0$ is not bounded).
Instead, we use quasi-compactness of $\tilde \L_{ib}$ (\ie $\sigma = 0$)
given by Remark~\ref{rem:imaginary_axis} and the continuity estimate 
of Proposition~\ref{prop-continuity}.
These together ensure  that the essential spectral radius of $\tilde \L_s$ is strictly 
less than $1$, and that the spectrum in a neighbourhood of $1$ contains
only isolated eigenvalues.
The rest of the argument goes exactly as \cite[Proof of Lemma 2.22]{AM}, distinguishing between
$b\neq 0$
the  and $b=0$.  In particular, proceeding as in \cite[Proof of Lemma 2.22]{AM}, we obtain the aperiodicity property and analyticity of the operator $Q_{ib}$ in the notation of \cite[Proof of Lemma 2.22]{AM}
in a neighborhood of $b$ for each $b\neq 0$. Also, in a neighborhood of $b=0$ we speak of the isolated eigenvalue $\lambda_{ib}$
(for the operator $\tilde \L_{ib}$) and corresponding spectral projection $P_{ib}$. Using again the continuity property of $\tilde \L_s$
given by  Proposition~\ref{prop-continuity}, we can continue $\lambda_s$ and $P_s$ in a neighborhood of $s=0$.
\\[4mm]
{\bf Acknowledgements:} We would like to thank Ian Melbourne for valuable discussions 
and Tomas Persson for informing us about reference \cite{AP}. 
We are also grateful for the support the Erwin Schr\"odinger Institute in Vienna, where this paper was completed.

\end{document}